\documentclass[12pt,reqno]{amsart}
\usepackage[pagewise]{lineno}\linenumbers
\usepackage{mathrsfs}
\usepackage{amssymb}
\usepackage{cases}
\usepackage{amsmath}
\usepackage{stmaryrd}
\usepackage{txfonts}
\usepackage{indentfirst,latexsym,bm}
\usepackage{graphicx}
\usepackage{psfrag}
\usepackage{amsmath,amssymb}
\usepackage{amsthm}
\usepackage[]{harpoon}
\usepackage[active]{srcltx} 
\usepackage[]{amsfonts}
\usepackage[]{amssymb}
\usepackage[]{harpoon}
\usepackage{epsfig}
\usepackage{caption}
\usepackage{epstopdf}
\usepackage{subfigure}
\usepackage{color}%
\begin{document}
\newtheorem{thm}{Theorem}[section]
\newtheorem{cor}[thm]{Corollary}
\newtheorem{lem}[thm]{Lemma}
\newtheorem{prop}[thm]{Proposition}
\theoremstyle{definition}
\newtheorem{defn}[thm]{Definition}
\theoremstyle{Assertion}
\newtheorem{asser}[thm]{Assertion}
\theoremstyle{remark}
\newtheorem{rem}[thm]{Remark}
\numberwithin{equation}{section}
\newcommand{\norm}[1]{\left\Vert#1\right\Vert}
\newcommand{\abs}[1]{\left\vert#1\right\vert}
\newcommand{\set}[1]{\left\{#1\right\}}
\newcommand{\Real}{\mathbb R}
\newcommand{\eps}{\varepsilon}
\newcommand{\To}{\longrightarrow}
\newcommand{\BX}{\mathbf{B}(X)}
\newcommand{\A}{\mathcal{A}}
\newcommand{\ts}{\textstyle}
\newcommand{\tg}{\mbox{\rm tg}}
\newcommand{\ctg}{\mbox{\rm ctg}}
\newcommand{\atg}{\tg^{-1}}
\newcommand{\actg}{\ctg^{-1}}
\newcommand{\asin}{\sin^{-1}}
\newcommand{\acos}{\cos^{-1}}
\newcommand{\dps}{\displaystyle}
\newcommand{\fnz}{\footnotesize}
\newcommand{\D}{\displaystyle}
\newcommand{\DF}[2]{\frac{\D#1}{\D#2}}
\newcommand{\scp}{\scriptstyle}

\title[A model for traffic flow]
{A model for traffic flow on a road with variable widths$^*$}%
\author[Sheng and Zhang]{Wancheng Sheng \qquad Qinglong Zhang}
\thanks{$^*$Supported by NSFC 11371240 and 11771274.\\
\indent Email: mathwcsheng@shu.edu.cn (Wancheng Sheng), zhangqinglong@shu.edu.cn (Qinglong Zhang)}
\dedicatory{Department of Mathematics, Shanghai University,
Shanghai, 200444, P.R.China
}

\begin{abstract}
We propose a model describing the traffic flow on a road with variable widths in this paper. The model, which is modified the Aw-Rascle model, is not conservative because of the source term. We obtain the elementary waves of the new traffic flow model, including rarefaction waves, shock waves, contact discontinuities and stationary waves. The Riemann problems of the system for the traffic flow are solved and some numerical results are given, which are almost the same as the theoretical ones.

\vskip 3pt

\noindent%
{\sc Keywords.}~ Traffic flow, variable widths, Aw-Rascle model, nonconservative system, elementary waves, Riemann problem.

\vskip 3pt
\noindent%
{\sc MSC2010.} Primary: 35L60, 35L65, 35L67, 35R03; Secondary: 76L05, 76N10.
\end{abstract}

\maketitle

\section{Introduction}
The Aw-Rascle (AR) model of traffic flow is given by \cite{AwRascle2}
\begin{equation}\label{1.1}
\left\{
\begin{array}{l}
\rho_t+(\rho u)_{x}=0,\\
(\rho (u+p))_t+(\rho u(u+p))_x=0,\\
\end{array}
\right.
\end{equation}
where $\rho, u$ represent the density and the velocity  of the traffic vehicle respectively. Here $p$ can be viewed as velocity offset and takes the form $p=\rho^{\gamma}~(\gamma>0)$ is analogous with the adiabatic gas constant in gas dynamics. The AR model describes the traffic flow in a unidirectional road, so here we assume that $\rho(t,x)\geq 0$ and $u(t,x)\geq 0$ basically.

The AR model is one of the main traffic flow models. It is proposed to resolve the theoretical inconsistencies of second order models pointed out by \cite{Daganzo5} and has been independently derived by Zhang in \cite{Zhang20}. AR model is also the basis for the multilane traffic flow model \cite{{Garavello6},{Greenberg7}},  and the hybrid traffic flow model \cite{MoutariRascle12}.

Aw and Rascle \cite{AwRascle2} investigated the Riemann problem of \eqref{1.1}. In \cite{Sun17}, Sun studied the Riemann problem and the interactions of elementary waves for AR model.
Modern transportation system is quite complicated because of non-isotropic. For example, on the highway, there will be multiple lanes, expand lanes, merging lanes, three or more lanes merging into two lanes or less lanes, etc. The feature of these phenomena is that the width of the interface of the road changes. The traffic flow is in a variable cross-section road.
In 2003, LeFloch etal \cite{LeflochThanh11} solved the Riemann problem of isentropic flow in a variable cross-section duct
\begin{equation}\label{1.2}
\left\{
\begin{array}{l}
(a\rho)_t+(a \rho u)_{x}=0,\\
(a\rho u)_t+(a(\rho u^2+p))_x=pa_x,\\
a_t=0,
\end{array}
\right.
\end{equation}
where $\rho, u, p$ represent the density, the velocity and the pressure of the fluid respectively, $a(x)$ represents the cross-section of the duct, which is independent of time. Recently, Sheng and Zhang studied the interaction of elementary waves for \eqref{1.2} including the stationary waves in \cite{Sheng16}.

Inspired by the model of fluid in a variable cross-section duct, in this paper, we propose a model which describes the traffic flows on a road with variable widths in section 2.  We do characteristic analysis in section 3. We give the elementary waves of the model, including the rarefaction waves, shock waves, contact discontinuities and stationary waves. To ensure the uniqueness of the stationary waves, we propose a global entropy condition similar as in \cite{LeflochThanh11,Sheng16}. In section 4, by using the characteristic analysis and phase plane analysis methods, we solve the Riemann problem for the system of traffic flow on a road with variable widths constructively. Some numerical tests are given in the last section, which are almost the same as the Riemann solutions we construct.

\section{Model for traffic flow  on a road with variable widths}
Assuming that traffic flow is in a unidirectional road. Denote $a(x)>0$ as the width of the road, which is a variable depending on $x$, see Fig.1.1. Here $\theta$ is the angle describing the change of $a(x)$, which changes very slowly. $x_1$ and $x_2$ are any two locations on the road with $x_1<x_2$. We will establish the model by the conservation of mass and momentum respectively.

\unitlength 0.8mm 
\linethickness{0.4pt}
\ifx\plotpoint\undefined\newsavebox{\plotpoint}\fi 
\begin{picture}(160.5,62)(-10,11)
\footnotesize
\put(23,51.75){\line(1,0){38}}
\multiput(61,51.75)(.087899543,.033675799){438}{\line(1,0){.087899543}}
\put(99.5,66.5){\line(1,0){43.75}}
\put(23.25,26.75){\line(1,0){119.75}}
\put(60.75,51.75){\line(1,0){11}}
\put(69.2,51.7){$\theta$}
\put(96.5,47){\vector(1,-2){.07}}\multiput(90,62.5)(.03367876,-.08031088){193}{\line(0,-1){.08031088}}
\put(90,62.75){\vector(1,0){7.75}}
\put(90,62.75){\vector(0,-1){16.25}}
\put(99,46){$P$}
\put(100.25,61.75){$Pa_x$}
\put(46,15){Fig.1.1.The traffic flow in a road with different widths.}
\put(54.75,51.5){\line(0,-1){24.75}}
\put(47,38.5){$a(x)$}
\put(29.75,43){\vector(1,0){12.75}}
\put(30,34.5){\vector(1,0){12.5}}
\put(112.25,52.75){\vector(1,0){20.75}}
\put(112.25,39){\vector(1,0){21.5}}
\qbezier(65.75,53.5)(65.75,52.88)(65.75,51.75)
\put(97.18,62.43){\line(0,-1){.938}}
\put(97.05,60.55){\line(0,-1){.938}}
\put(96.93,58.68){\line(0,-1){.938}}
\put(96.8,56.8){\line(0,-1){.938}}
\put(96.68,54.93){\line(0,-1){.938}}
\put(96.55,53.05){\line(0,-1){.938}}
\put(96.43,51.18){\line(0,-1){.938}}
\put(96.3,49.3){\line(0,-1){.938}}
\put(96.18,47.43){\line(-1,0){.929}}
\put(94.32,47.36){\line(-1,0){.929}}
\put(92.47,47.29){\line(-1,0){.929}}
\put(90.61,47.22){\line(-1,0){.929}}
\put(85.5,65.5){\line(0,-1){4.5}}
\put(92.75,67.25){\line(0,-1){3.25}}
\put(86,66){$\Delta x$}
\put(63.25,26.65){\circle*{1}}
\put(105,26.65){\circle*{1}}
\put(63.75,23.25){$x_1$}
\put(105.5,23.25){$x_2$}
\end{picture}

\noindent {\bf The conservation of mass}

As in the Lighthill-Whitham-Richards (LWR) theory, the traffic flow is considered as a compressible fluid. We consider the mass conservation in [$x_1,x_2$] within the time period [$t_1,t_2$]. It follows that
\begin{equation}\label{1.3}
\begin{array}{l}
\displaystyle \int_{x_1}^{x_2}{a\rho(t_2,x)}dx-\int_{x_1}^{x_2}{a\rho(t_1,x)}dx=\int_{t_1}^{t_2}{a\rho u(t,x_1)}dt-\int_{t_1}^{t_2}{a\rho u(t,x_2)}dt,
\end{array}
\end{equation}
under the assumption that the variables are continuously differentiable, we have
\begin{equation}\label{1.4}
\begin{array}{l}
(a\rho)_t+(a \rho u)_{x}=0.
\end{array}
\end{equation}

\noindent {\bf The conservation of momentum}

In the AR traffic flow model \eqref{1.1}, the pressure $p=\rho^{\gamma}$ is regarded as the velocity offset, which can be called ``pseudo-velocity", i.e., the cars accelerate when the pressure from the front decreases, and vice versa. In fact, if we denote the convective derivative ${\rm d}/{\rm d}t:=\partial_t+u\partial_x$, the second equation of \eqref{1.1} can be rewritten as ${\rm d}(u+p)/{\rm d}t=0$, which expresses the fact that $u+p$ remains constant along the direction of the car moving.

We consider the momentum conservation in [$x_1,x_2$] within the time period [$t_1,t_2$]. We have that
\begin{equation}\label{1.5}
\begin{array}{l}
\displaystyle \underbrace{\int_{x_1}^{x_2}{a\rho (u+p)(t_2,x)}dx-\int_{x_1}^{x_2}{a\rho (u+p)(t_1,x)}dx}_A=\\[12pt]
\displaystyle \underbrace{\int_{t_1}^{t_2}{a\rho u(u+p)(t,x_1)}dt-\int_{t_1}^{t_2}{a\rho u(u+p)(t,x_2)}dt}_B+\underbrace{\int_{t_1}^{t_2}\int_{x_1}^{x_2}{\rho upa_x}dxdt}_C,
\end{array}
\end{equation}
where $A$ represents the change of the momentum during the time period [$t_1,t_2$] in [$x_1,x_2$]. $B$ is the change of momentum in the two locations $x_1$ and $x_2$ respectively during the period [$t_1,t_2$]. $C$ is the momentum offset which is obtained because of the changing width of the road. Here $\rho u$ is the mass of vehicles, $pa_x$ is the velocity offset, which is positive as the width increases $(a_x>0)$ and negative as the width decreases $(a_x<0)$.

For the derivation of $C$, we consider the infinitesimal $\Delta x$ in [$x_1,x_2$], see Fig.1.1. The pressure offset obtained in $\Delta x$ is
$$
\Delta p= p\cdot sin\theta \cdot\Delta x\approx p\cdot tan\theta \cdot\Delta x= p\cdot a_x \cdot\Delta x,
$$
which is reasonable when $\theta$ is small. Similarly, we use the pressure to represent the pseudo-velocity. Thus the momentum offset corresponding to $\Delta x$  within time $\Delta t$ is
\begin{equation}\label{1.6}
\displaystyle \Delta I =m\cdot  \Delta v=a\rho u \cdot  \Delta t \cdot \frac{\Delta p}{a} =a\rho u \cdot  \frac{p\cdot a_x \cdot\Delta x}a\Delta t=\rho u pa_x \Delta x \Delta t.
\end{equation}
Integrating \eqref{1.6} in $[t_1,t_2]\times[x_1,x_2]$, we have the expression of $C$.

Under the assumption that the variables are continuously differentiable, \eqref{1.5} is equivalent to
\begin{equation}\label{1.7}
\begin{array}{l}
(a\rho (u+p))_t+(a \rho u (u +p))_x=\rho u pa_x.
\end{array}
\end{equation}

Now we give our full model as
\begin{equation}\label{1.8}
\left\{
\begin{array}{l}
(a\rho)_t+(a \rho u)_{x}=0,\\[5pt]
(a\rho (u+p))_t+(a \rho u (u +p))_x=\rho u pa_x,\\[5pt]
a_t=0,
\end{array}
\right.
\end{equation}
where $\rho, u$ represent the density and the velocity  of the traffic vehicles as before, $p=\rho^{\gamma}$ is given as the velocity offset, $a(x)$ represents the width of the road.

The first two equations represent the conservation of mass and momentum respectively, which is explained above. Usually $a(x)$ is given as a prior,  here we view it as a variant which is independent of time,  which is the third equation of \eqref{1.8}. We see that \eqref{1.8} is a non-conservative system, the definition of weak solutions can't be derived as usual. For more details about the nonconservative systems, we refer to \cite{{IsaacsonTemple8},{LeflochThanh11}}.

\section{Characteristic analysis for system \eqref{1.8}}
\subsection{Preliminaries}
Denote $U=(u,\rho,a)$. Considering a smooth solution, system \eqref{1.8} can be rewritten as
\begin{equation}\label{2.1}
A(U)\partial_t U+B(U)\partial_x U=0,
\end{equation}
where
\begin{equation}\label{2.2}
A=\left(
\begin{array}{cccc}
 0 & a & 0  \\
a  & a\gamma \rho^{\gamma-1} &  0  \\
0 &0 & 1
\end{array}
 \right),
\quad  B=\left(
\begin{array}{cccc}
a\rho & au &   \rho u  \\
au & au \gamma\rho^{\gamma-1} &  -u\rho^{\gamma}  \\
0 &0 & 0
\end{array}
 \right).
\end{equation}
The equation \eqref{2.1} has three eigenvalues
\begin{equation}\label{2.3}
\lambda_1=u-\gamma \rho^{\gamma},\ \ \ \lambda_2=u,\ \ \ \lambda_3=0.
\end{equation}
 The corresponding right eigenvectors are
$$
\vec{r}_1=(-\gamma \rho^{\gamma-1},1,0)^T,~~\vec{r}_2=(0,1,0)^T,\\
\vec{r}_3=\left(u(1+\gamma)\rho^{\gamma},-(\rho u+\rho^{\gamma+1}),
a(u-\gamma\rho^{\gamma}) \right)^T.
$$
The 2- and 3-characteristic fields are linearly degenerate, while the 1-characteristic fields are genuinely nonlinear. System \eqref{2.1} is non-strictly hyperbolic because $\lambda_1,\lambda_2$  may coincide with $\lambda_3$. More precisely, setting
$$
\Gamma_+:u=\gamma \rho^{\gamma},\quad\Gamma_0: u=0,
$$
we have that
$$
\lambda_1=\lambda_3\quad{\rm on} \quad\Gamma_{+},\quad
\lambda_2=\lambda_3\quad{\rm on}\quad\Gamma_{0}.
$$

Since the model describes a traffic flow in a unidirectional road, we consider the elementary waves of \eqref{2.1} in the first quarter of the $(u,\rho)$ plane. Based on the above discussion, we use the curve $\Gamma_+$ to separate the first quarter into two regions. For convenience, we will view them as $D_{1}$ and $D_2$:
\begin{equation}\label{2.4}
\begin{array}{cc}
D_1  =\Big\{(u,\rho)\big|0<u<\gamma \rho^{\gamma}~\Big\}\quad{\rm and}\quad
D_2  =\Big\{(u,\rho)\big|u>\gamma \rho^{\gamma}~\Big\}.
\end{array}
\end{equation}
In either of the two regions, system \eqref{2.1} is strictly hyperbolic and we have
$$
\lambda_1<\lambda_3<\lambda_2 \ \ {\rm in} \ \ D_1,\quad
\lambda_2>\lambda_1>\lambda_3\ \ {\rm in}\ \ D_2.
$$
\subsection{The rarefaction waves}
We look for self-similar solutions $U(\xi)=(u,\rho,a)(\xi), \xi=x/t$. The Riemann invariant $w_i$ corresponding to $\lambda_i$ should satisfy
$$
\triangledown w_i \cdot \vec{r}_i=0,\quad i=1,2.
$$
The Riemann invariants along the $\lambda_1$ characteristic field are
\begin{equation}\label{2.7}
w_1^1=u+\rho^{\gamma}\quad {\rm and}\quad  w_1^2=a.
\end{equation}
The Riemann invariants along the $\lambda_2$ characteristic field are
\begin{equation}\label{2.6'}
w_2^1=u,\quad {\rm and}\quad w_2^2=a.
\end{equation}

From \eqref{2.7}, we have that $a(x)$ remains constant across rarefaction wave. Therefore system \eqref{1.8} degenerates to system \eqref{1.1}.

For a given left-hand state $U_0=(u_0,\rho_0,a_0)$, we determine the right-hand state $U=(u,\rho,a)$, which can be connected to $U_0$ by rarefaction curves as
\begin{equation}\label{2.8}
R(U,U_0):\left\{
\begin{array}{l}
\xi=\lambda_1=u-\gamma \rho^{\gamma},\\
u-u_0=-\rho^{\gamma}+\rho_0^{\gamma},\\
\rho<\rho_0,u>u_0.
\end{array}
\right.
\end{equation}
\subsection{The discontinuity solutions}
For a discontinuity at $\xi=\sigma$, the Rankine-Hugoniot relation associated with the last equation of \eqref{1.8} is that
\[-\sigma[a]=0,\]
where $[a]:~=a_1-a_0$ is the jump of the cross-section $a$. We have the following conclusions:\\
1)~~$[a]=0:$~the cross-section~$a$~remains constant across the non-zero speed shocks;\\
2)~~$\sigma=0:$~the shock speed vanishes, here we assume $[a]\not=0$ and call it stationary contact discontinuity.

From case 1), the Rankine-Hugoniot conditions corresponding to \eqref{1.8} is given by
\begin{equation}\label{2.9}
\left\{\begin{array}{ll}
-\sigma[\rho]=[\rho u],\\[5pt]
-\sigma[\rho(u+p)]=[\rho u(u+p)],
\end{array}\right.
\end{equation}
where $[f]=f_1-f_0$ is the jump of the function $f$, $\sigma$ is the speed of propagation of the discontinuity connecting the states $U_0$ and $U_1$.

By solving \eqref{2.9}, we obtain the contact discontinuity $J(U,U_0)$
$$
J(U,U_0):u=u_0=\sigma,\rho\neq \rho_0.
$$

For a non-zero speed shock wave. Assume that $U_0$ is a given left-hand state, the shock curve $S(U,U_0)$ consisting of all right-hand states $U$ satisfying Lax shock condition \cite{Lax9} is
\begin{equation}\label{2.10}
S(U,U_0):\left\{
\begin{array}{l}
\displaystyle \sigma=u+\frac{\rho_0(\rho_0^{\gamma}-\rho^{\gamma})}{\rho-\rho_0},\\[8pt]
u-u_0=-\rho^{\gamma}+\rho_0^{\gamma},\\[2pt]
\rho>\rho_0,u<u_0.
\end{array}
\right.
\end{equation}

Here we see that the shock wave curves coincide with the rarefaction wave curves in the phase plane. Moreover, we have
$$
\frac{\rm{d}u}{\rm{d}\rho}=-\gamma \rho^{\gamma-1}<0,\ \ \frac{\rm{d}^2 u}{\rm{d} \rho^2}=-\gamma(\gamma-1) \rho^{\gamma-2}<0.
$$
So both the rarefaction wave curve and shock wave curve are decreasing and concave in the $(u,\rho)$ plane.

Next we turn to the stationary contact discontinuity. As in \cite{LeflochThanh11}, a stationary solution is independent of time, thus we search for time-independent smooth solutions of the ordinary differential equations
\begin{equation}\label{2.11}
\left\{
\begin{array}{ll}
\displaystyle (a\rho u)'=0,\\
(a\rho u(u+p))'=\rho u p a'.
\end{array}
\right.
\end{equation}
where $(f)'$ denotes $\frac{{\rm{d}}f}{{\rm d}x}$, etc. By solving \eqref{2.11}, we have the following lemma.
\begin{lem}
For smooth solutions, system \eqref{2.11} is equivalent to
\begin{equation}\label{2.12}
\left\{
\begin{array}{ll}
\displaystyle a_0 \rho_0 u_0=a\rho u,\\[3pt]
\displaystyle u_0^{\frac{\gamma}{1+\gamma}}\left(\rho_0^{\gamma}+\frac{\gamma}{1+2\gamma}u_0\right)=
u^{\frac{\gamma}{1+\gamma}}\left(\rho^{\gamma}+\frac{\gamma}{1+2\gamma}u\right),
\end{array}
\right.
\end{equation}
where $(u_0,\rho_0,a_0)$ and $(u,\rho,a)$  represent the left-hand state and right-hand state of the stationary discontinuity respectively.
\end{lem}
\begin{proof}
If we give the left-hand state $(u_0,\rho_0,a_0)$, the first equation of \eqref{2.11} is equivalent to
$$
a\rho u=a_0\rho_0 u_0.
$$
Differential the above equation with respect to $x$, we have
\begin{equation}\label{2.13}
a'=-\frac{a\rho'}{\rho}-\frac{au'}{u}.
\end{equation}
Substituting \eqref{2.13} into the second equation of \eqref{2.11}, we get
\begin{equation}\label{2.14}
\displaystyle \left(1+\frac{\rho^{\gamma}}{u}\right)u'+(1+\gamma)\rho^{\gamma-1}\rho'=0,
\end{equation}
which is equivalent to
\begin{equation}\label{2.15}
\frac{\rm{d}\rho^{\gamma}}{\rm{d}u}=-\frac{\gamma}{u(1+\gamma)}\rho^{\gamma}-\frac{\gamma}{1+\gamma}.
\end{equation}
By solving \eqref{2.15}, we have
\begin{equation}\label{2.16}
u^{\frac{\gamma}{1+\gamma}}\left(\rho^{\gamma}+\frac{\gamma}{1+2\gamma}u\right)=C,
\end{equation}
where $C$ is a constant, which is given by
$$
C=u_0^{\frac{\gamma}{1+\gamma}}\left(\rho_0^{\gamma}+\frac{\gamma}{1+2\gamma}u_0\right).
$$
It follows the lemma.
\end{proof}
If we fix the left-hand state $(u_0,\rho_0,a_0)$ across the stationary contact discontinuity. From  \eqref{2.12}, we have
\begin{equation}\label{2.17}
\begin{array}{l}
\displaystyle F(\rho):=k_0\rho^{\gamma+1}-k_1\rho^{\frac{1+2\gamma}{1+\gamma}}+k_2=0,
\end{array}
\end{equation}
where
\begin{equation}\label{2.17'}
k_0=\left(\frac{a_0\rho_0u_0}{a}\right)^{\frac{\gamma}{1+\gamma}},
k_1=u_0^{\frac{\gamma}{1+\gamma}}(\rho_0^{\gamma}+\frac{\gamma}{1+2\gamma}u_0),
k_2=\frac{\gamma}{1+2\gamma}\left(\frac{a_0\rho_0u_0}{a}\right)^{\frac{1+2\gamma}{1+\gamma}}.
\end{equation}
It is obviously that
\begin{equation}\label{2.18}
\lim_{\rho\to 0}{F(\rho)}=k_2>0,~~~\lim_{\rho\to +\infty}{F(\rho)}=+\infty.
\end{equation}
Furthermore, we have that
\begin{equation}\label{2.19}
\displaystyle \frac{\rm{d}F(\rho)}{\rm{d}\rho}=k_0(\gamma+1)\rho^{\gamma}-k_1\frac{1+2\gamma}{1+\gamma}\rho^{\frac{\gamma}{1+\gamma}}.
\end{equation}
So
\begin{equation}\label{2.20}
\frac{\rm{d}F(\rho)}{\rm{d}\rho}\left\{\begin{array}{ll}
<0,&\rho<\rho_{min},\\[5pt]
>0,&\rho>\rho_{min},
\end{array}\right.
\end{equation}
where
$$
\rho_{min}=\left(\frac{k_1(1+2\gamma)}{k_0(\gamma+1)^2}\right)^{\frac{1+\gamma}{\gamma^2}}.
$$
From  \eqref{2.17'}, we know that $k_0$, $k_1$ are determined by $U_0$ and $a$. So if $U_0$ is given, $\rho_{min}$ depends only on $a$.  From \eqref{2.18}, \eqref{2.17} admits a solution if and only if
$$
F(\rho_{min})\leq0.
$$
We assume $F(\rho_{min})<0$ here and after, then $F(\rho)=0$ admits exactly two values $\rho_*(U_0)<\rho_{min}<\rho^{*}(U_0)$, such that
\begin{equation}\label{2.21}
F(\rho_*(U_0))=F(\rho^{*}(U_0))=0.
\end{equation}
We have the following lemma.
\begin{lem}
Given the left-hand state $U_0=(u_0,\rho_0,a_0)$, there exits a stationary contact discontinuity connecting to the right-hand state $U=(u,\rho,a)$ if and only if $F(\rho_{min})\leq0$.
More precisely, \\
1) If  $F(\rho_{min})>0$, there are no stationary contacts;\\
2) If $F(\rho_{min})\leq0$, there are exactly two points $U_*=(u_*,\rho_*,a),~~U^{*}=(u^{*},\rho^{*},a)$ satisfying
$$
F(\rho_*;U_0)=F(\rho^{*};U_0)=0,
$$
where $\rho_*(U_0)<\rho_{min}<\rho^{*}(U_0)$, moreover, the two values $U_*$ and $U^{*}$ coincide if $F(\rho_{min})=0$.
\end{lem}
We denote the stationary contact discontinuity by $S_0(U,U_0)$, the states $U_*=(u_*,\rho_*,a)$ and $U^*=(u^{*},\rho^{*},a)$ have the following properties.
\begin{lem}
Given the left-hand state $(u_0,\rho_0,a_0)$, the two states $U_*=(u_*,\rho_*,a)$ and $U^{*}=(u^{*},\rho^{*},a)$ that can be connected  by $S_0(U;U_0)$ satisfing $U^{*}\in D_1,~U_*\in D_2$.
\end{lem}
\begin{proof}
From the first equation of \eqref{2.12}, the conclusion $u_*>0, u^*>0$ can be easily derived.
It is sufficient to prove that $u_*>\gamma \rho_*^{\gamma},~u^*<\gamma \rho^*$. We prove $u_*>\gamma\rho_*^{\gamma}$. From the above discussion, we have
\begin{equation}\label{2.22}
\rho_*<\rho_{min}=\left(\frac{k_1(1+2\gamma)}{k_0(\gamma+1)^2}\right)^{\frac{1+\gamma}{\gamma^2}}.
\end{equation}
It follows that
$$
\rho_*^{\frac{\gamma^2}{\gamma+1}}<\frac{1+2\gamma}{(\gamma+1)^2}
u_0^{\frac{\gamma}{1+\gamma}}\left(\rho_0^{\gamma}+\frac{\gamma}{1+2\gamma}u_0\right)
\left(\frac{a}{a_0\rho_0u_0}\right)^{\frac{\gamma}{\gamma+1}}.
$$
Combining with \eqref{2.12}, we have $u_*>\gamma\rho_*^{\gamma}$. For the proof of $u^*<\gamma\rho^*$, it is similar, we omit here. Thus proving the lemma.
\end{proof}
As in \cite{Han,LeflochThanh11,Sheng16}, the Riemann problem for \eqref{1.2} may admit up to a one-parameter family of solutions. This phenomenon can be avoided by requiring Riemann
solutions to satisfy an \emph{admissibility criterion}: monotone condition on the component $a$. Following  \cite{Andrianov,LeflochThanh11}, we
impose the following global entropy condition on stationary contact discontinuity.

\noindent%
{\bf Global entropy condition.} Given $U_0=(u_0,\rho_0,a_0)$, the state $a$ is implicitly determined by $\rho$ from \eqref{2.17}. The width $a=a(\rho)$ obtained from \eqref{2.17} is a monotone function of $\rho$ along the stationary curve $S_0(U,U_0)$ in the $(u,\rho)$-plane.

Under the global entropy condition, we call the stationary contact discontinuity as stationary wave in traffic flow. Under the transformation
$$
x\mapsto -x,~~u\mapsto -u,
$$
a right-hand state $U=(u,\rho,a)$ becomes a left-hand state $U=(-u,\rho,a)$. Without loss of generality, we assume $a$ is an increasing function of $\rho$ from now on, i.e. $a_0<a$.
\begin{lem}
Global entropy condition is equivalent to the statement that any
stationary wave has to remain in the closure of only one domain $D_i, i=1,2$, see Fig.2.1.
\end{lem}
\begin{proof}
Given the left-hand state $(u_0,\rho_0,a_0)$, we differential \eqref{2.12} and get
\begin{equation}\label{2.23}
\left\{\begin{array}{ll}
\displaystyle \frac{{\rm{d}}a}{a}+\frac{\rm{d}\rho}{\rho}+\frac{{\rm{d}}u}{u}=0,\\[8pt]
\displaystyle \frac{{\rm{d}}u}{\rm{d}\rho}=-\frac{u(\gamma+1)\rho^{\gamma-1}}{u+\rho^{\gamma}}.
\end{array}\right.
\end{equation}
We substitute the first equation into the second one and have
\begin{equation}\label{2.24}
\left\{\begin{array}{ll}
\displaystyle \frac{{\rm{d}}u}{{\rm{d}}a}=\frac{u(\gamma+1)\rho^{\gamma}}{a(u-\gamma \rho^{\gamma})},\\[9pt]
\displaystyle \frac{\rm{d}\rho}{{\rm{d}}a}=-\frac{\rho(u+\rho^{\gamma})}{a(u-\gamma\rho^{\gamma})}.
\end{array}\right.
\end{equation}
Since $a$ is increasing, so ${\rm{d}}a>0$. We conclude that $u_0>\gamma \rho_0^{\gamma} (u_0<\gamma \rho_0^{\gamma})$ if and only if $u_*>\gamma \rho_{*}^{\gamma}(u^*<\gamma (\rho^{*})^{\gamma})$.
Thus we prove the lemma.
\end{proof}

\unitlength 1mm 
\linethickness{0.4pt}
\ifx\plotpoint\undefined\newsavebox{\plotpoint}\fi 
\begin{picture}(80.75,59.75)(-40,3)
\footnotesize
\put(13.5,16.75){\vector(1,0){58.25}}
\put(15.25,14.25){\vector(0,1){45.5}}
\put(12.5,58){$\rho$}
\put(75.5,15.75){$u$}
\qbezier(15.25,16.75)(32.2,43.63)(51.5,53)
\put(55,54.25){$\Gamma_+:u=\gamma\rho^{\gamma}$}
\put(16,7){Fig.2.1. The curve of $S_0$ in $(u,\rho)$ plane.}
\put(33,39.5){\circle{1}}
\put(29,35.25){\circle{1}}
\qbezier(32.6,40)(25.75,49.63)(15.25,51.75)
\qbezier(29.5,34.9)(38.25,29.75)(39,16.75)
\put(25.75,46.95){\circle*{1}}
\put(23,22){$S_0(U,U_0)$}
\put(36.75,26.5){\circle*{1}}
\put(18,52){$a>a_0$}
\put(29,46){$a<a_0$}
\put(22.75,43.25){$U_0$}
\put(35,33.6){$a<a_0$}
\put(39,26.5){$U_0$}
\put(41,21.5){$a>a_0$}
\end{picture}

In order to construct the Riemann solution, we have to know when do the shock waves coincide with the stationary waves? that is when does the shock speed equals to zero?

From \eqref{2.10}, we have
\begin{equation}\label{2.25}
\begin{array}{ll}
\displaystyle \sigma =\sigma(U,U_0)=u+\frac{\rho_0(\rho_0^{\gamma}-\rho^{\gamma})}{\rho-\rho_0}
=u_0+\frac{\rho(\rho_0^{\gamma}-\rho^{\gamma})}{\rho-\rho_0},\ \ \rho>\rho_0.
\end{array}
\end{equation}
Let $\sigma=0$, which yields
$$
u_0=\rho \frac{\rho^{\gamma}-\rho_0^{\gamma}}{\rho-\rho_0}>\gamma \rho_0^{\gamma}.
$$
It is equivalent to $U_0\in D_2$. To be more explicitly, we have the following lemma.
\begin{lem}
The shock speed $\sigma(U,U_0)$ may change its sign along the shock curve $S(U,U_0)$, more precisely, \\
1)~ If $U_0\in D_1$, then $\sigma(U,U_0)$ remains negative
\begin{equation}\label{2.260}
\sigma(U,U_0)<0.
\end{equation}
2)~ If $U_0\in D_2$, then there exists a point $U=\widetilde{U}_0 \in D_1$ which is on the shock curve $S(U,U_0)$ such that
\begin{equation}\label{2.26}
\left\{\begin{array}{lll}
\displaystyle \sigma(\widetilde{U}_0,U_0)=0,\\[5pt]
\displaystyle  \sigma(U,U_0)>0,~~\rho \in(\rho_0,\widetilde{\rho}_0),\\[5pt]
\displaystyle  \sigma(U,U_0)<0,~~\rho \in(\widetilde{\rho}_0,+\infty).
\end{array}\right.
\end{equation}
\end{lem}

\section{The initial value problem of \eqref{1.8}}
Firstly, we consider the initial value problem of \eqref{1.8} with 
\begin{equation}\label{4.1'}
U=(u,\rho,a)=\left\{
\begin{array}{lll}
U_-=(u_-,\rho_-,a_0), \qquad x<x_1,\\[5pt]
U_0(x)=(u(x),\rho (x),a(x)), \qquad x_1<x<x_2,\\[5pt]
U_+=(u_+,\rho_+,a_1), \qquad x>x_2,
\end{array}
\right.
\end{equation}

\begin{figure}[htbp]
\begin{minipage}[t]{0.5\textwidth}
\centering
\includegraphics[width=\textwidth]{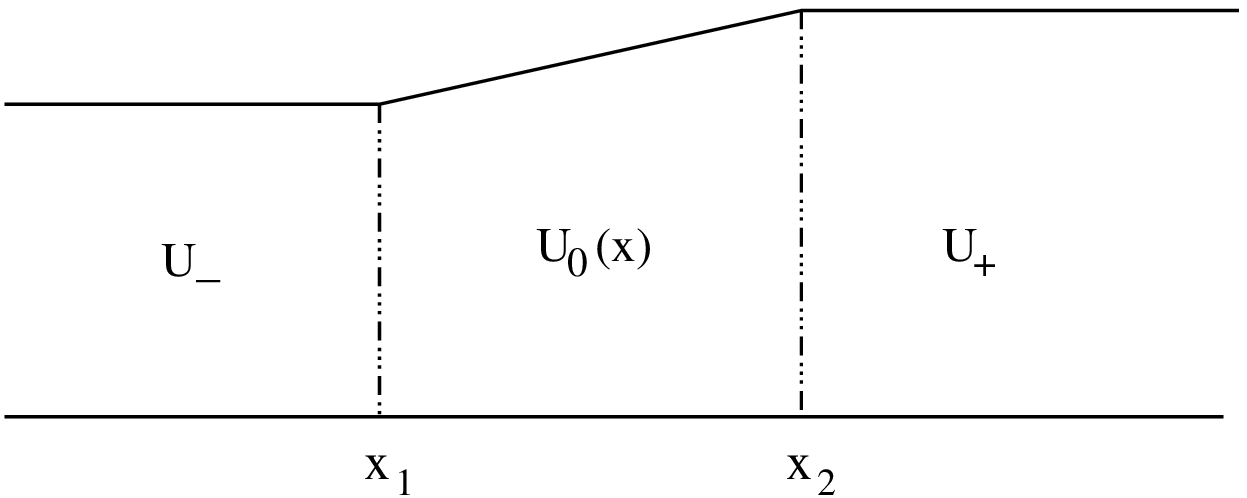}
\end{minipage}
\caption*{Fig.4.1. The initial value problem of \eqref{1.8}.}
\end{figure}
See Fig. 4.1. Solving \eqref{4.1'} is very complicated, because the wave patterns of \eqref{4.1'} involving the rarefaction waves, shock waves, contact discontinuities, stationary waves and their interactions. To make things simpler, we consider the large time behavior of the initial value problem \eqref{4.1'}. So there is no wave interaction involved, which is basically the Riemann problem of \eqref{1.8} (\cite{Liu}),  see Fig. 4.2. The only difference is that in the original initial problem \eqref{4.1'}, the states changes continuously across the stationary wave between $x_1$ and $x_2$. While in the Riemann problem, the state jumps across the stationary wave. Thus study the Riemann problem will give a direct understanding of initial value problem \eqref{4.1'}. We left the interaction of elementary waves to the future consideration.

\begin{figure}[htbp]
\begin{minipage}[t]{0.5\textwidth}
\centering
\includegraphics[width=\textwidth]{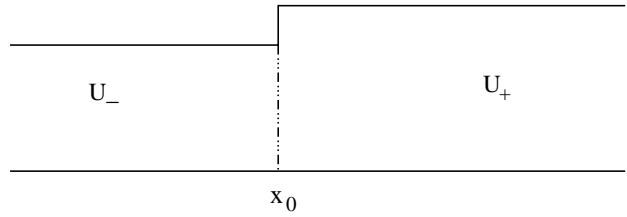}
\end{minipage}
\caption*{Fig.4.2. The initial value problem of \eqref{1.8}.}
\end{figure}

Considering the Riemann problem of equations \eqref{1.8} with initial data
\begin{equation}\label{2.5}
(u,\rho,a)=\left\{
\begin{array}{lll}
(u_-,\rho_-,a_0),~~~~x <0,\\
(u_+,\rho_+,a_1),~~~~x>0.
\end{array}
\right.
\end{equation}
The Riemann solutions admits the self-similar ones, which depends on $\xi=\frac{x}{t}$. Thus, the Riemann problem is reduced to the boundary value problem at infinity as follows
\begin{equation}\label{2.6}
\left\{
\begin{array}{lll}
-\xi(a\rho)_\xi+(a\rho u)_\xi=0,\\[5pt]
-\xi(a\rho(u+p))_\xi+(a\rho u(u+p))_\xi-\rho upa_\xi=0,\\[5pt]
-\xi a_\xi=0,  \\[5pt]
(u,\rho,a)(\pm \infty)=(u_\pm,\rho_\pm,a_{0,1}).
\end{array}
\right.
\end{equation}
In this section we establish the global existence of the Riemann problem for \eqref{1.8} with \eqref{2.5}. We will use the following notations:

(i) $S(U_j,U_i)\oplus R(U_k,U_j)$ indicates that there is a shock wave from the left-hand state $U_i$ to the right-hand state $U_j$, followed by a rarefaction wave from the left-hand state $U_j$ to the right-hand state $U_k$.

(ii) $J(U,U_m)$ denotes the contact discontinuity from the left-hand state $U_m=(u_m,\rho_m,a)$ to the right-hand state $U=(u,\rho,a)$, here we assume that $u=u_m,~\rho\neq \rho_m$.

We will construct the Riemann solutions case by case.

\noindent%
{\bf Case 1.} $U_-\in D_2$. First we define two points $U_{-*}\in D_2$ and $\widetilde{U}_-\in D_1$, where $U_{-*}\in S_0(U,U_-)$, $\widetilde{U}_-\in S(U,U_-)$ at which the shock speed vanishes,  in view of lemma 3.5.
We draw the curve $R(U,U_{-*})$, which intersects with u-axis at $\overline{U}=(\bar{u},0)$, where $\bar{u}=u_{-*}+\rho_{-*}^{\gamma}$. Similarly from lemma 3.5, we have $\widetilde{U}_{-*}\in S(U,U_{-*})$ at which the shock speed vanishes. See Fig.3.1. The relative position of $\widetilde{U}_-$ and $\widetilde{U}_{-*}$ is obtained in Lemma A. We discuss them as follows.

\begin{figure}[htbp]
\begin{minipage}[t]{0.4\textwidth}
\centering
\includegraphics[width=\textwidth]{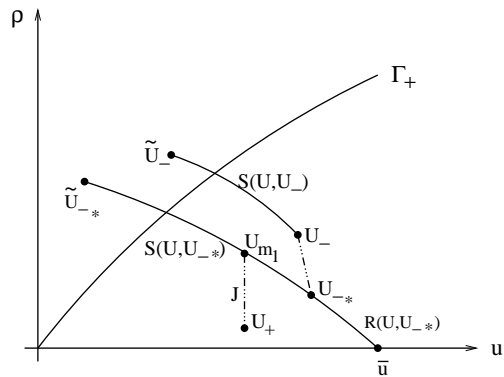}
\end{minipage}
\caption*{Fig.3.1. Case 1. $U_-\in D_2$.}
\end{figure}

\begin{figure}[htbp]
\subfigure{
\begin{minipage}[t]{0.35\textwidth}
\centering
\includegraphics[width=\textwidth]{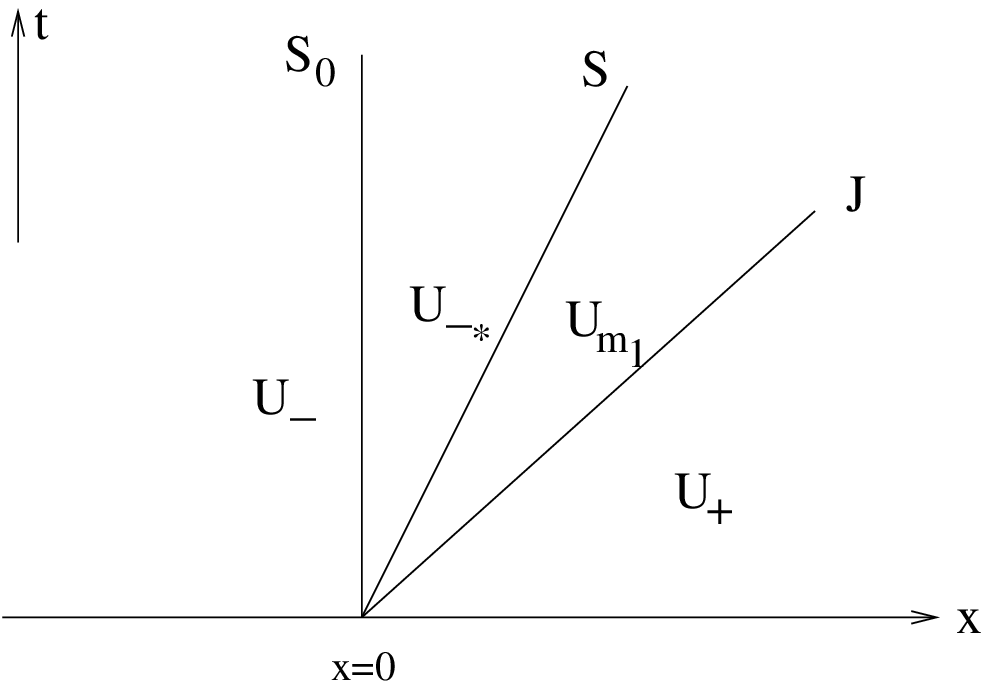}
\end{minipage}
}
\subfigure{
\begin{minipage}[t]{0.35\textwidth}
\centering
\includegraphics[width=\textwidth]{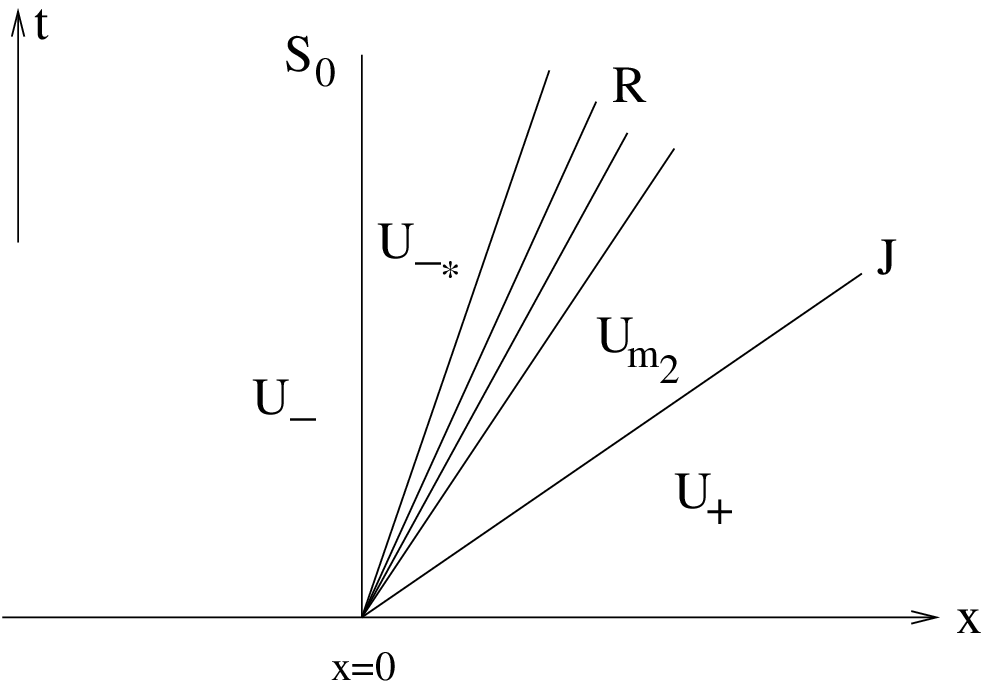}
\end{minipage}
}
\caption*{Fig.3.2. Case 1.1-3. The solutions are $S_0\oplus S({\rm or}\ R)\oplus J$.}
\end{figure}

\noindent%
{\bf Case 1.1.}  $\widetilde{u}_{-*}<u_+<u_{-*}$. Denote $J(U_+,U)\cap S(U,U_{-*})=\{U_{m_1}\}$, see Fig.3.2 (left). the Riemann solution is
$$
S_0(U_{-*},U_-)\oplus S(U_{m_1},U_{-*})\oplus J(U_+,U_{m_1}).
$$
{\bf Case 1.2.} $u_{-*}<u_+<\bar{u}$. Denote $J(U_+,U)\cap R(U,U_{-*})=\{U_{m_2}\}$, see Fig.3.2 (right). We have the Riemann solution as
$$
S_0(U_{-*},U_-)\oplus R(U_{m_2},U_{-*})\oplus J(U_+,U_{m_2}).
$$
\noindent%
{\bf Case 1.3.} $u_+>\bar{u}$, there exists a vacuum in this case. Let $J(U_+,U)\cap\{\rho=0\}=\{M\}$, then the Riemann solution is
$$
S_0(U_{-*},U_-)\oplus R(\overline{U},U_{-*})\oplus {\rm{Vacuum}} \oplus J(U_+,M).
$$

\noindent%
{\bf Case 1.4.} $0<u_+<\tilde{u}_{-}^{*}$.  Let $\widetilde{U}_-^{*}\in D_1$ which is jumped by $\widetilde{U}_-$ with stationary wave.  From any point $U_0\in S(U,\widetilde{U}_-), \rho_0>\tilde{\rho}_-$, a stationary wave jumps from $U_0$ to some state $U_0^{*}\in D_1$, which states form the curve $S_0(U_0^{*},U_0)$, see Fig.3.3. To be precise, setting
$$
S_0(U,U_0): =\big\{ U: \exists S_0(U,U_0)~{\rm from}~a_0~{\rm to}~a_1. U_0\in S(U,\widetilde{U}_-),~\rho\geq \tilde{\rho}_-^{*} \big\}.
$$
Denote $J(U_+,U)\cap S_0(U_0^{*},U_0)=\{U_{m_3}^{*}\}$, $U_{m_3}^{*}$ is jumped by stationary wave $S_0(U_{m_3}^{*},U_{m_3}), U_{m_3}\in S(U,U_-)$, see Fig.3.3. The Riemann solution in this case is
$$
S(U_{m_3},U_-)\oplus S_0(U_{m_3}^{*},U_{m_3}) \oplus J(U_+,U_{m_3}^{*}).
$$

\begin{figure}[htbp]
\subfigure{
\begin{minipage}[t]{0.35\textwidth}
\centering
\includegraphics[width=1.1\textwidth]{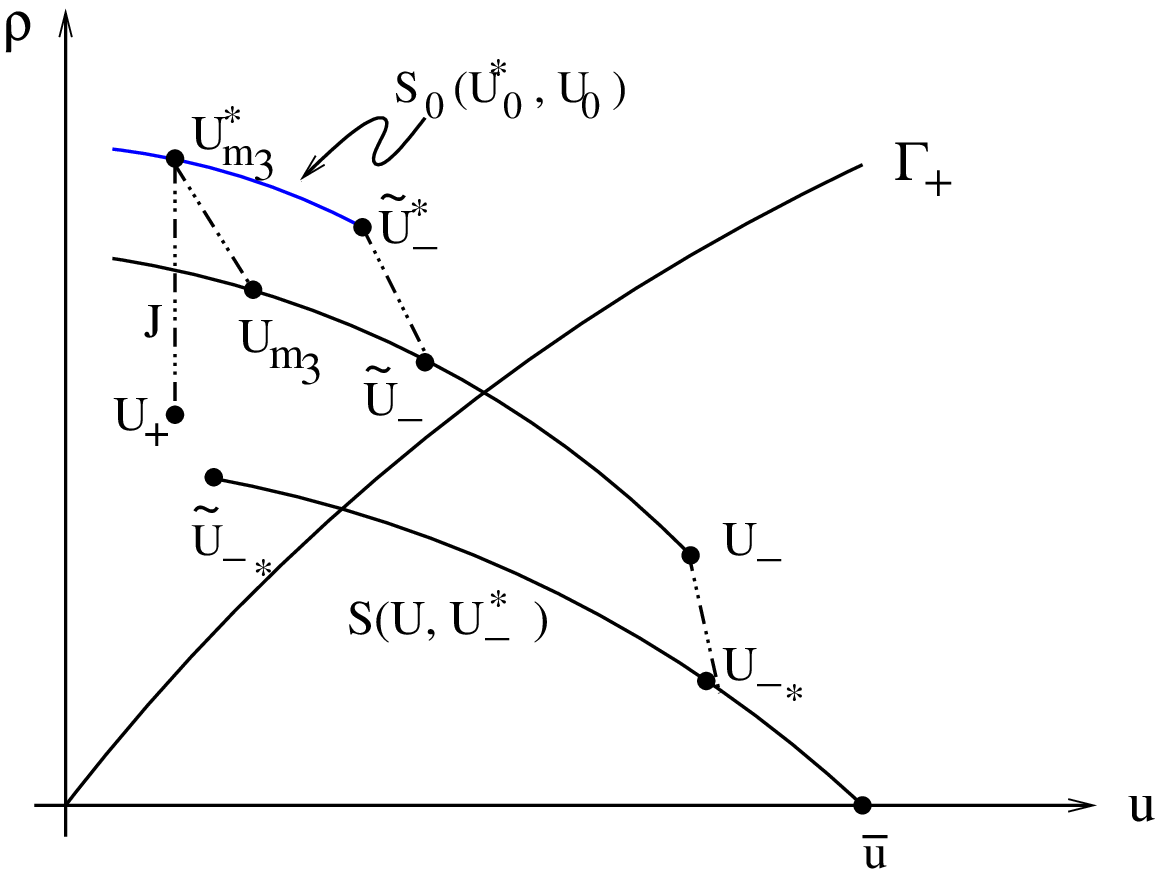}
\end{minipage}
}
\subfigure{
\begin{minipage}[t]{0.35\textwidth}
\centering
\includegraphics[width=0.9\textwidth]{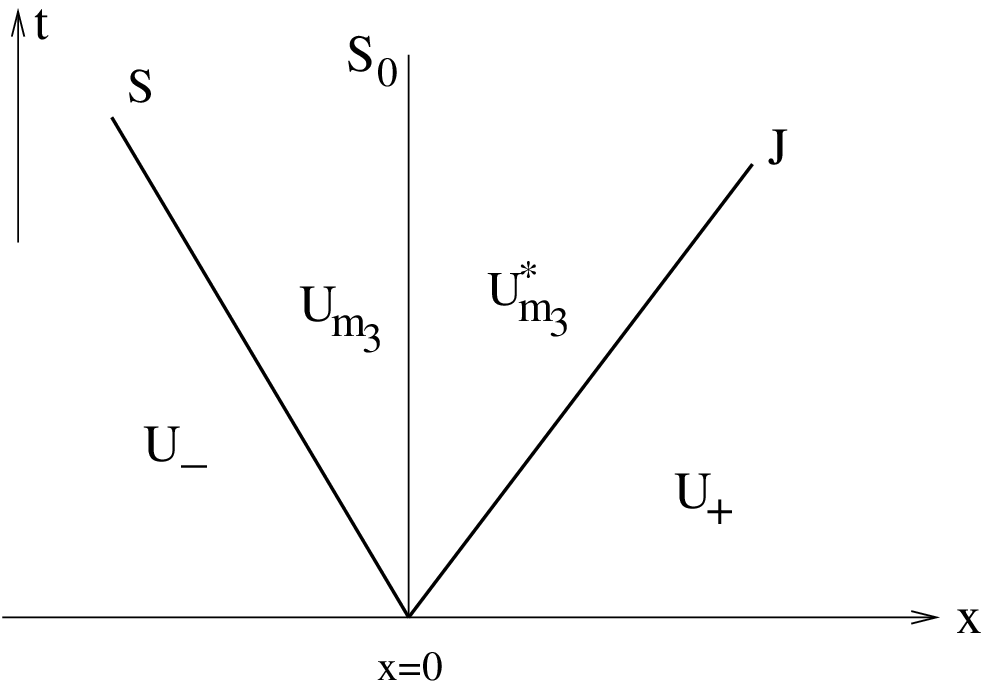}
\end{minipage}
}
\caption*{Fig.3.3. Case 1.4. $U_-\in D_2$, $0<u_+<\tilde{u}_{-*}$.}
\end{figure}

\begin{figure}[htbp]
\begin{minipage}[t]{0.45\textwidth}
\centering
\includegraphics[width=\textwidth]{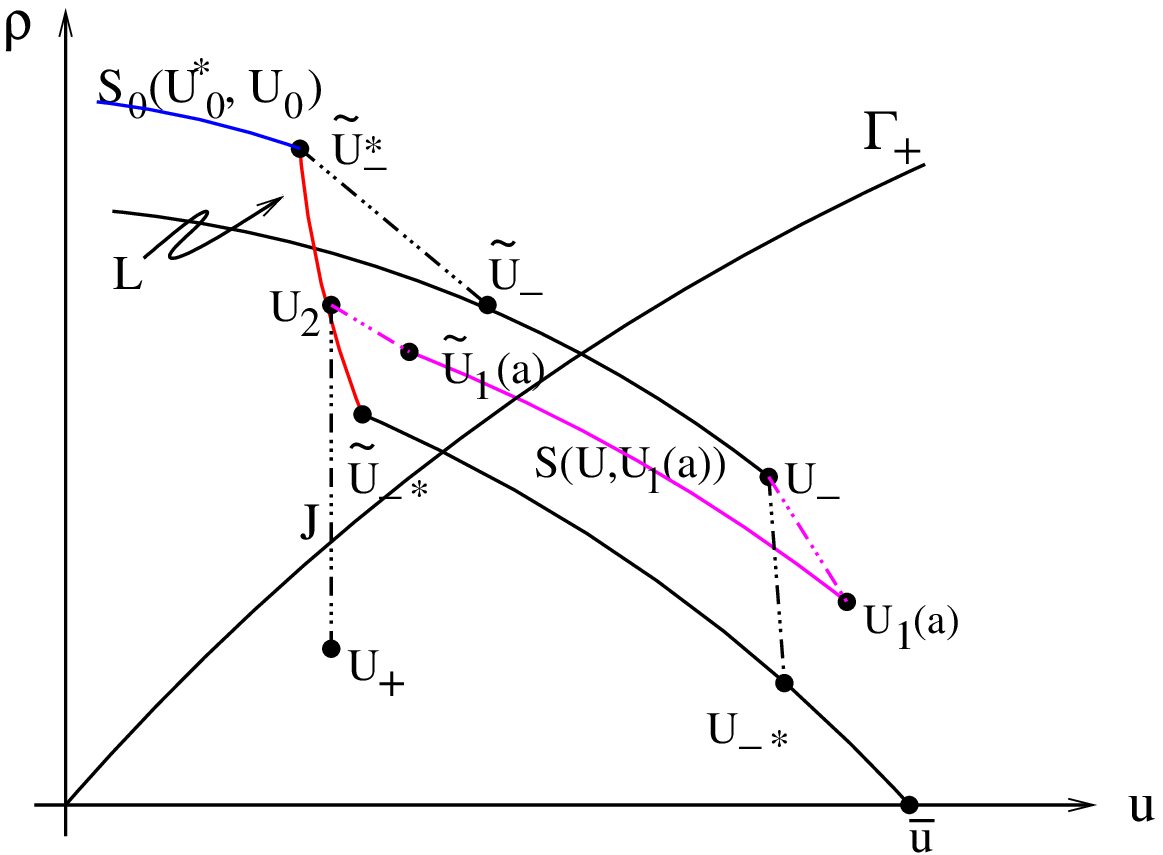}
\end{minipage}
\caption*{Fig.3.4. Three waves with the same zero speed.}
\end{figure}

\noindent%
{\bf Case 1.5.} From lemma \ref{B1}, we have $\tilde{u}_-^{*}<\tilde{u}_{-*}$ as $a_1>a_0$. This case holds when $\tilde{u}_-^{*}<u_+<\tilde{u}_{-*}$. We consider a solution containing three waves with the same zero speed.  See Fig.3.4. First $U_-$ jumps to $U_1(a)$ by stationary wave $S_0(U,U_-)$ with $a$ shifting from $a_0$ to an intermediate cross-section area $a\in[a_0,a_1]$, then $U_1(a)$ jumps to $\widetilde{U}_1(a)\in D_1$ by $S(\widetilde{U}_1(a),U_1(a))$ with $\sigma(\widetilde{U}_1(a),U_1(a))=0$. Finally $\widetilde{U}_1(a)$ jumps to $U_2$ by a stationary wave to shift the cross-section area $a$ to $a_1$.
Set
$
L=\{U(a)|a\in [a_0,a_1]\},
$
whenever
$$
J(U_+,U)\cap L\neq \emptyset,
$$
the solution contains three discontinuities having the same zero speed
$$
S_0(U_1(a),U_-)\oplus S(\widetilde{U}_1(a),U_1(a)) \oplus S_0(U_2,\widetilde{U}_1(a)) \oplus J(U_+,U_2).
$$

\noindent%
{\bf Remark:} When $a_1<a_0$, we have $\tilde{u}_-^{*}>\tilde{u}_{_*}$ in view of lemma \ref{A1}. Thus the curve $S_0(U_0^{*},U_0)\cup L\cup S(U,U_{-*})$ is folding.  There are at most three solutions when $\tilde{u}_-^{*}>u_+>\tilde{u}_{_*}$. Since the solution with standing wave is unstable in this case \cite{Liu}, we still have two solutions. How to select the physical meaning solution needs further research.

\noindent%
{\bf Case 2.} $U_-\in D_1$. In this case, both the rarefaction wave and shock wave have negative speed. First we define $\{U_C\}=R(U,U_-)\cap \Gamma_+$. A stationary wave jumps from $U_C$ to ${U_{C}}_*\in D_2$.
We draw the curve $R(U,U_{C*})$, which intersects with u-axis at $\overline{U}=(\bar{u},0)$. Similarly from lemma 3.5, we have $\widetilde{U}_{C}^{*}\in S(U,U_{C*})$ at which the shock speed vanishes. See Fig.3.5. We discuss them as follows.

\begin{figure}[htbp]
\begin{minipage}[t]{0.45\textwidth}
\centering
\includegraphics[width=\textwidth]{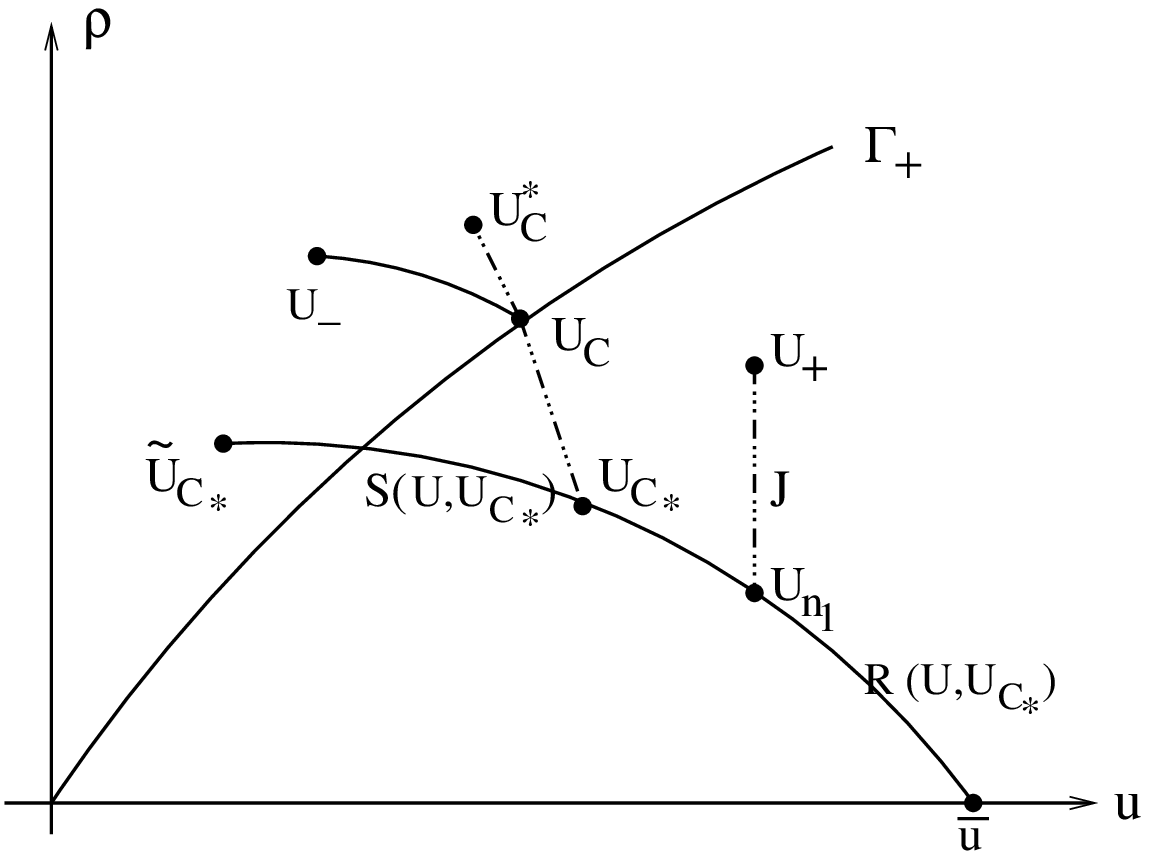}
\end{minipage}
\caption*{Fig.3.5. Case 2. $U_-\in D_1$.}
\end{figure}

\begin{figure}[htbp]
\subfigure{
\begin{minipage}[t]{0.35\textwidth}
\centering
\includegraphics[width=\textwidth]{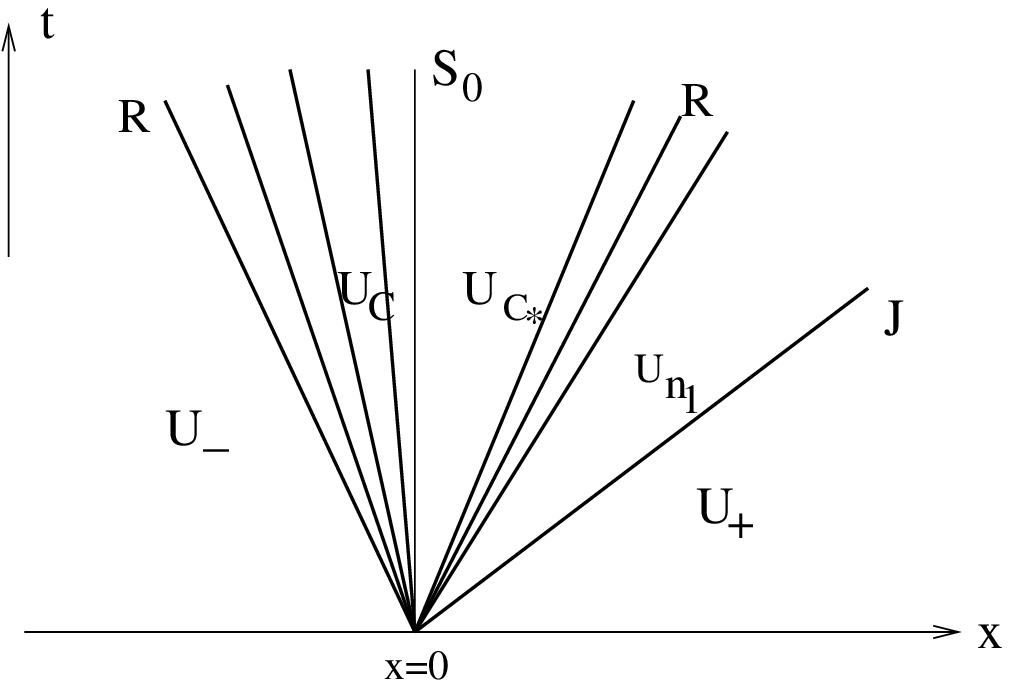}
\end{minipage}
}
\subfigure{
\begin{minipage}[t]{0.35\textwidth}
\centering
\includegraphics[width=\textwidth]{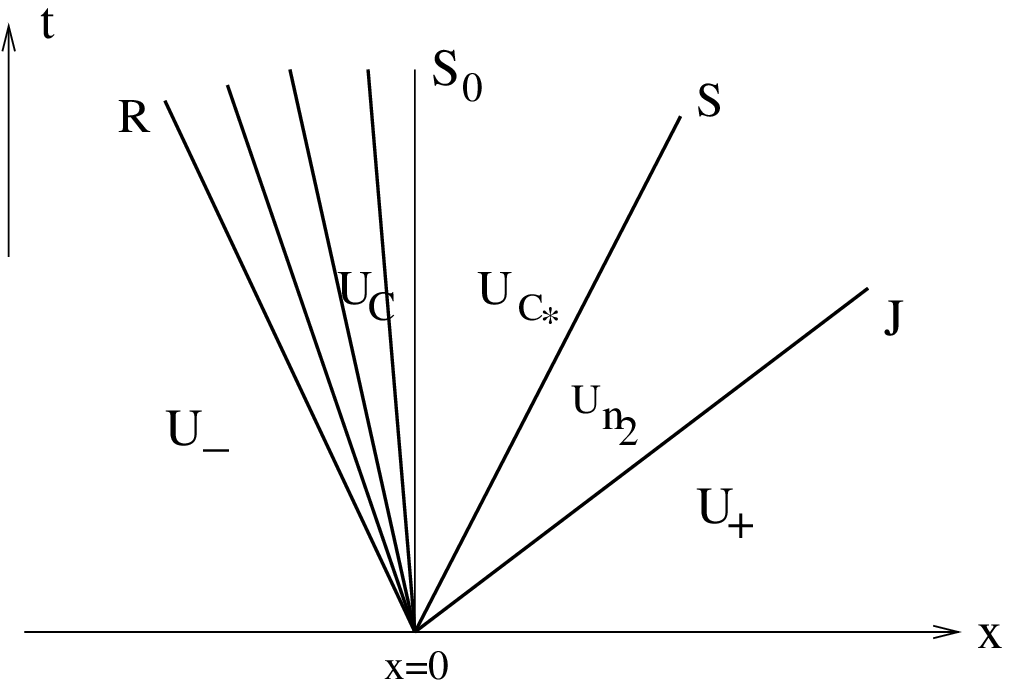}
\end{minipage}
}
\caption*{Fig.3.6. The solutions are $R\oplus S_0\oplus R({\rm or} S)\oplus J$.}
\end{figure}

\noindent%
{\bf Case 2.1.} $u_{c_*}<u_+<\bar{u}$. Denote $J(U_+,U)\cap R(U,U_{c_*})=\{U_{n_1}\}$, see Fig.3.6 (left). We have the Riemann solution as
$$
R(U_C,U_-)\oplus S_0(U_{C*},U_C)\oplus R(U_{n_1},U_{C*})\oplus J(U_+,U_{n_1}).
$$

\noindent%
{\bf Case 2.2.} $\tilde{u}_{c}^{*}<u_+<u_{c_*}$. Denote $J(U_+,U)\cap S(U,U_{-*})=\{U_{n_2}\}$, see Fig.3.6 (right). the Riemann solution is
$$
R(U_C,U_-)\oplus S_0(U_{C*},U_C)\oplus S(U_{n_2},U_{C*})\oplus J(U_+,U_{n_2}).
$$

\noindent%
{\bf Case 2.3.} $u_+>\bar{u}$, there exists a vacuum in this case. Let $J(U_+,U)\cap\{\rho=0\}=\{M\}$, then the Riemann solution is
$$
R(U_C,U_-)\oplus S_0(U_{C*},U_C)\oplus R(\overline{U},U_{C*})\oplus {\rm{Vacuum}} \oplus J(U_+,M).
$$

\noindent%
{\bf Case 2.4.} This construction holds when $U_-$ passes through a shock wave first. See Fig.3.7. Let $U_C^{*}\in D_1$ which is jumped by $U_C$ with stationary wave.  From any point $U_0\in S(U,U_-)\cup R(U,U_-), \rho_0>\rho_c$, a stationary wave jumps from $U_0$ to some state $U_0^{*}\in D_1$, such states form the curve $S_0(U_0^{*},U_0)$.

\noindent%
1) If $0<u_+<u_-$, see Fig.3.7. Denote $J(U_+,U)\cap S_0(U_0^{*},U_0)=\{U_{k_1}^{*}\}$, $U_{k_1}^{*}$ is jumped by stationary wave $S_0(U_{k_1}^{*},U_{k_1})$, where$U_{k_1}\in S(U,U_-)$. The Riemann solution in this case is (see Fig.3.8(left))
$$
S(U_{k_1},U_-)\oplus S_0(U_{k_1}^{*},U_{k_1}) \oplus J(U_+,U_{k_1}^{*}).
$$
2) If $u_-<u_+<u_c$. Denote $J(U_+,U)\cap S_0(U_0^{*},U_0)=\{U_{k_2}^{*}\}$, $U_{k_2}^{*}$ is jumped by stationary wave $S_0(U_{k_2}^{*},U_{k_2})$, where  $U_{k_2}\in R(U,U_-)$. The Riemann solution in this case is (see Fig.3.8(right))
$$
R(U_{k_2},U_-)\oplus S_0(U_{k_2}^{*},U_{k_2}) \oplus J(U_+,U_{k_2}^{*}).
$$

\begin{figure}[htbp]
\begin{minipage}[t]{0.45\textwidth}
\centering
\includegraphics[width=\textwidth]{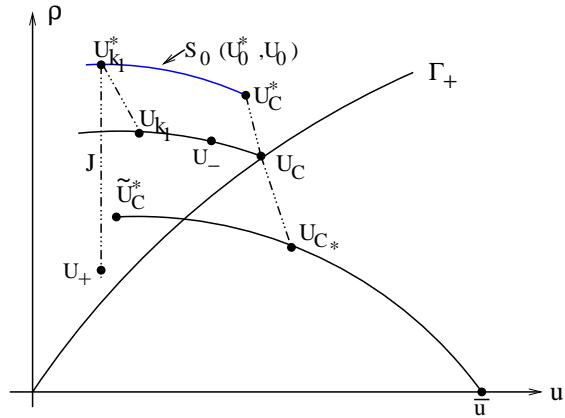}
\end{minipage}
\caption*{Fig.3.7. Case 2.4. $U_-\in D_1$.}
\end{figure}

\begin{figure}[htbp]
\subfigure{
\begin{minipage}[t]{0.35\textwidth}
\centering
\includegraphics[width=\textwidth]{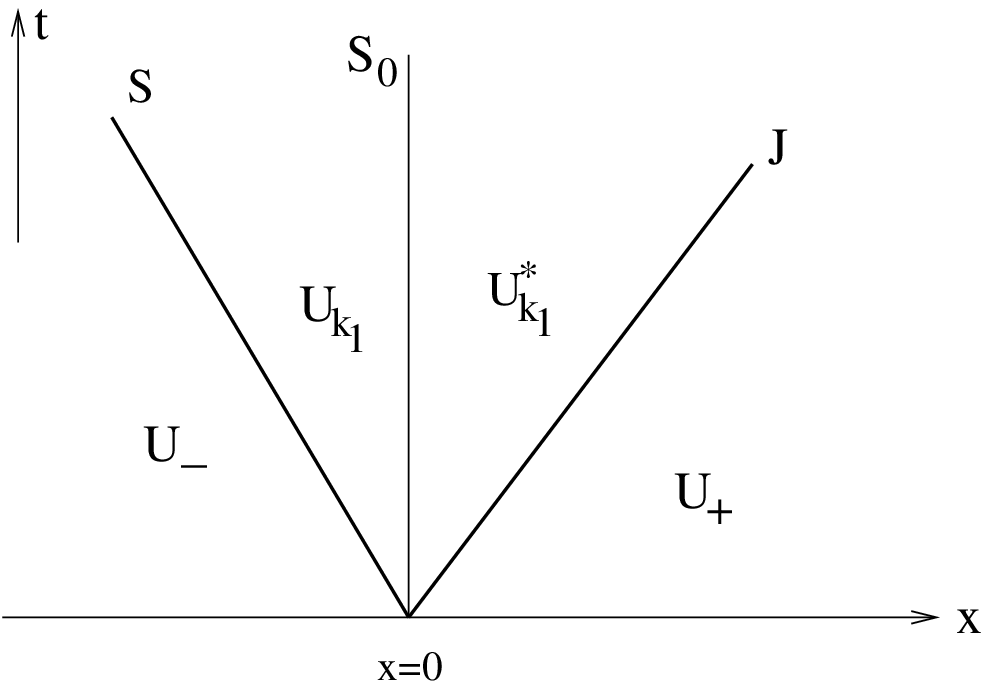}
\end{minipage}
}
\subfigure{
\begin{minipage}[t]{0.35\textwidth}
\centering
\includegraphics[width=\textwidth]{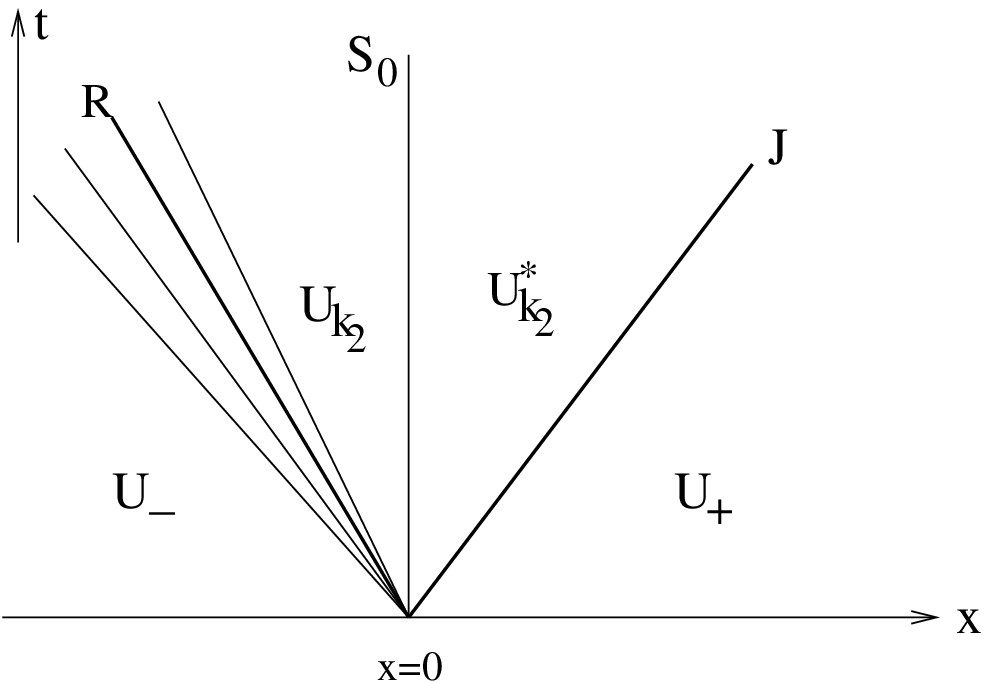}
\end{minipage}
}
\caption*{Fig.3.8. The solutions are $R({\rm or} S)\oplus S_0\oplus J$.}
\end{figure}

\begin{figure}[htbp]
\begin{minipage}[t]{0.45\textwidth}
\centering
\includegraphics[width=\textwidth]{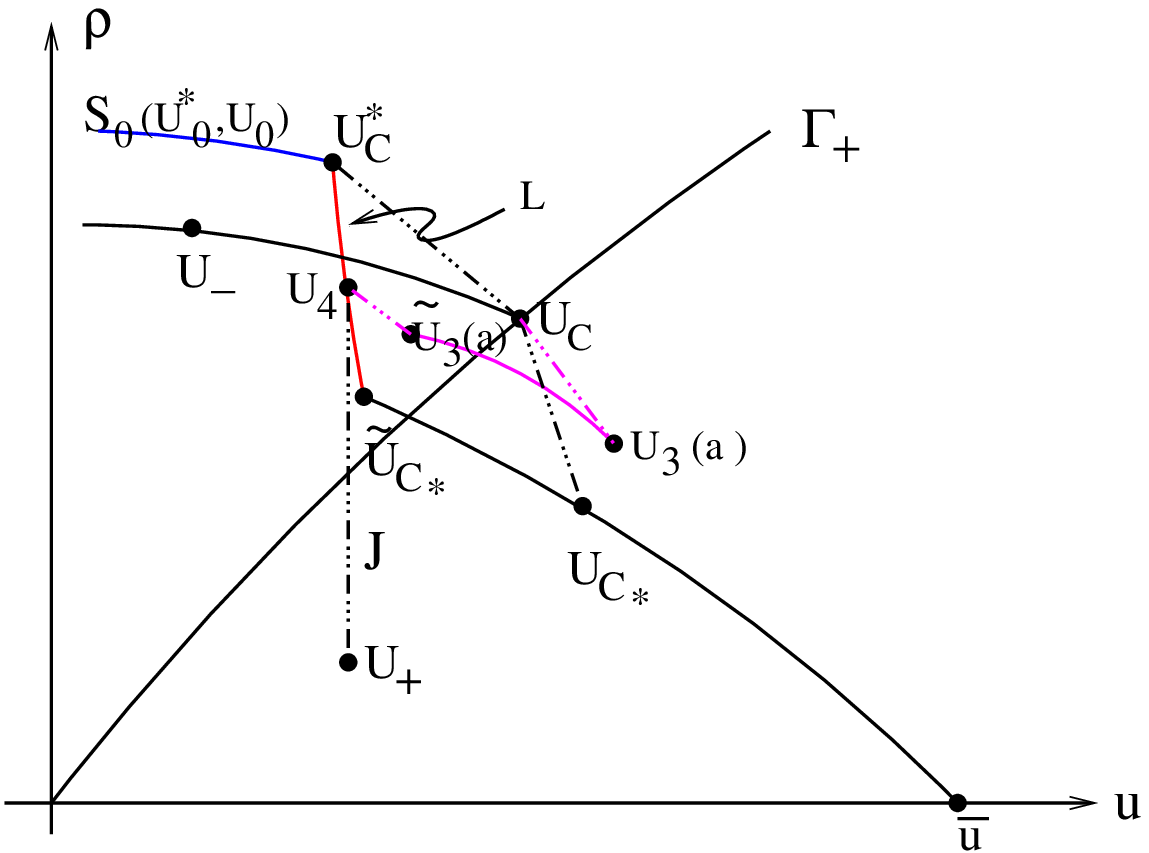}
\end{minipage}
\caption*{Fig.3.9. Three waves with the same zero speed.}
\end{figure}

\noindent%
{\bf Case 2.5.}  We have $u_c^{*}<\tilde{u}_{c*}$ as $a_1>a_0$ from lemma A. This case holds when $u_c^{*}<u_+<\tilde{u}_{c*}$. We consider a solution containing three waves with the same zero speed.   See Fig.3.9. First $U_-$ touches to $U_C$ by rarefaction wave $R(U_C,U_-)$, followed by a stationary wave $S_0(U_3(a),U_C)$ with $a$ shifting from $a_0$ to an intermediate cross-section area $a\in[a_0,a_1]$, then $U_3(a)$ jumps to $\widetilde{U}_3(a)\in D_1$ by $S(\widetilde{U}_3(a),U_3(a))$ with $\sigma(\widetilde{U}_3(a),U_3(a))=0$. Finally $\widetilde{U}_3(a)$ jumps to $U_4$ by a stationary wave to shift the cross-section area $a$ to $a_1$.
Set
$
L=\{U(a)|a\in [a_0,a_1]\},
$
whenever
$$
J(U_+,U)\cap L\neq \emptyset,
$$
the solution contains three discontinuities having the same zero speed
$$
R(U_C,U_-)\oplus S_0(U_3(a),U_C)\oplus S(\widetilde{U}_3(a),U_3(a)) \oplus S_0(U_4,\widetilde{U}_3(a)) \oplus J(U_+,U_4).
$$

\section{The Numerical Simulations}
In this section, we give some numerical simulations of the solutions for system \eqref{1.8}. The discrete scheme we use here is the Lax-Friedrichs scheme. The time step $\triangle t$ and the mesh size $\triangle x$ are given as uniform. Setting $x_j=j\triangle x, j\in Z$, and $t_n=n\triangle t, n\in N$. Define
$$
\lambda=\frac{\triangle t}{\triangle x},
$$
we have the standard difference scheme for \eqref{1.8}
\begin{equation}\label{4.1}
U_j^{n+1}=U_j^{n}-\lambda\big(g(U_j^{n},U_{j+1}^{n})-g(U_{j-1}^{n},U_{j}^{n})\big)+\frac{\lambda}2\big(0,(\rho up)_j^{n}(a_{j+1}-a_{j-1})\big),
\end{equation}
where $U=(u,\rho,a)$, $g(U,V)$ and $f(U)$ are defined by
\begin{equation}\label{4.2}
\left\{\begin{array}{ll}
\displaystyle g(U,V):=\frac12\big(f(U)+f(V)\big)-\frac1{2\lambda}\big(V-U\big),\\[6pt]
f(U):=\big(a\rho, a\rho (u+p)\big).
\end{array}\right.
\end{equation}
The constant $\lambda$ should satisfy the CFL stability condition
\begin{equation}\label{4.3}
\lambda~\max \limits_{U} \big|f'(U)\big| \leq1.
\end{equation}
Before given the numerical examples, we notice that in traffic flow problems, velocity $u$ is a decreasing function of density $\rho$. It describes that the vehicles will accelerate when the density is low, and vice versa. The computation domain we set is $[0,20]$. The CFL number is 0.5 in the following simulations. We use 2000 grid points. $a(x)$ is jumping from $a_0$ to $a_1$ in each case. Recall that $U=(u,\rho,a)$. In the first two examples, we consider the initial value problem \eqref{4.1'}. i.e., the widths change continuously. While in the other examples, we mainly consider the Riemann problem \eqref{2.5}.

Although  the angle $\theta={\rm arctan}\frac{a_1-a_0}{x_2-x_1}$ (see Fig.1.1) has the limit value $\pi/2$ in the jumping case, we note that our model is still meaningful by solving Riemann problems. Since the expression of stationary wave \eqref{2.12} still holds when we take the limit $x_2\rightarrow x_1$. Similar results can be seen in \cite{Andrianov,LeflochThanh11} for duct flows.

\noindent%
{\bf Example 1.} (see Fig. 5.1) The initial data is given as
\begin{equation}\label{5.1-1}
U(x,0)=\left\{\begin{array}{ll}
U_-=(1.5,0.5,2.0),\quad 0<x<3,\\[6pt]
U_0(x)=(1.5,0.5,a(x)),\quad 3<x<8,\\[6pt]
U_+=(0.75,1.0,3.0), \quad  8<x<20.
\end{array}\right.
\end{equation}
Here $a(x)$ is 
\begin{equation}\label{5.1-2}
a(x)=\left\{\begin{array}{ll}
2.0,\quad 0<x<3,\\[6pt]
0.2\cdot x+1.4,\quad 3<x<8,\\[6pt]
3.0, \quad 8<x<20.
\end{array}\right.
\end{equation}

\begin{figure}[htbp]
\begin{minipage}[t]{\textwidth}
\centering
\includegraphics[width=\textwidth]{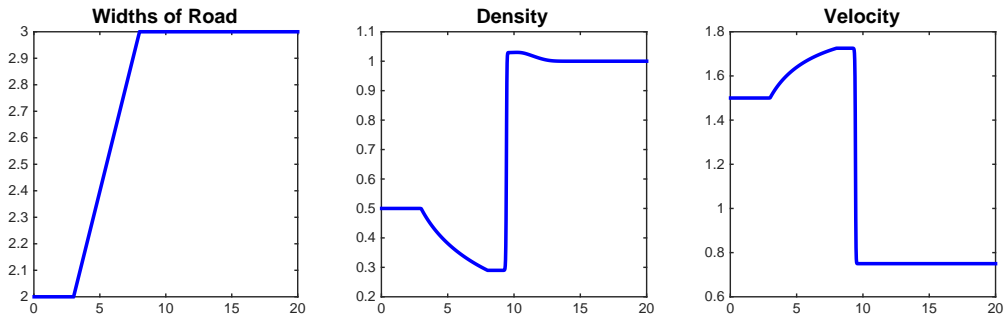}
\end{minipage}
\caption*{Fig. 5.1. $S_0\oplus S\oplus J$.}
\end{figure}

The numerical result is shown at time $t=6.0s$ in Fig. 5.1. The solution begins with a stationary wave $S_0$ in the domain $[3,8]$, followed by a shock wave $S$, then followed by a contact discontinuity $J$. It is the same as in the case 1.1.

\noindent%
{\bf Example 2.} (see Fig. 5.2) The initial data is given as
\begin{equation}\label{5.1-3}
U(x,0)=\left\{\begin{array}{ll}
U_-=(3.0,1.0,2.0),\quad 0<x<3,\\[6pt]
U_0(x)=(3.0,1.0,a(x)),\quad 3<x<8,\\[6pt]
U_+=(1.365,1.65,2.5), \quad  8<x<20.
\end{array}\right.
\end{equation}
Here $a(x)$ is 
\begin{equation}\label{5.1-4}
a(x)=\left\{\begin{array}{ll}
2.0,\quad 0<x<3,\\[6pt]
0.1\cdot x+1.7,\quad 3<x<8,\\[6pt]
2.5, \quad 8<x<20.
\end{array}\right.
\end{equation}

\begin{figure}[htbp]
\begin{minipage}[t]{\textwidth}
\centering
\includegraphics[width=\textwidth]{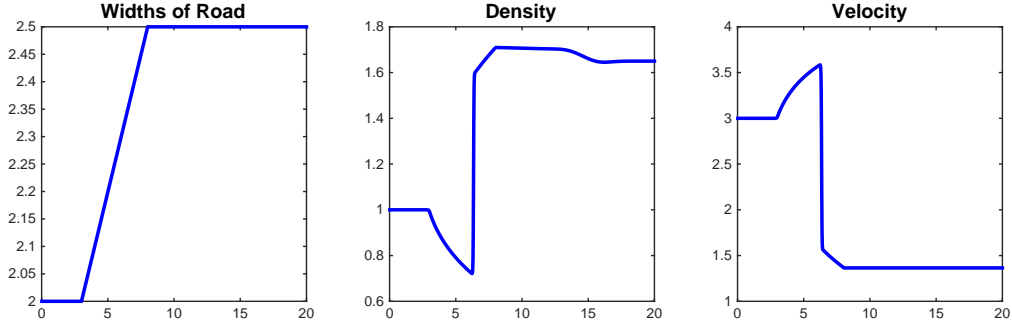}
\end{minipage}
\caption*{Fig. 5.2. $S_0\oplus S\oplus S_0\oplus J$.}
\end{figure}

The numerical result is shown at time $t=6.0s$ in Fig. 5.2. The solution includes three waves with the same zero speed. We have the a stationary wave $S_0$, followed by a standing shock wave $S$, followed by another stationary wave, then followed by a contact discontinuity $J$. It is the same as in the case 1.5.

\noindent%
{\bf Example 3.} (see Fig. 5.3) The initial data is given as

\begin{equation}\label{5.3-1}
U(x,0)=\left\{\begin{array}{ll}
U_-=(1.5, 0.5, 2.0)\in D_2,\quad 0<x<6,\\[6pt]
U_+=(0.75,1.0,  3.0)\in D_1, \quad  6<x<20,
\end{array}\right.\qquad \gamma=2.0.
\end{equation}

The numerical result is shown at time $t=7.5s$ in Fig. 5.3. The solution begins with a stationary wave $S_0$, followed by a shock wave $S$, then followed by a contact discontinuity $J$. Compared this case with example 1, we can see that the stationary jump of Riemann problem is exactly the limit of the standing wave in example 1.

\begin{figure}[htbp]
\begin{minipage}[t]{\textwidth}
\centering
\includegraphics[width=\textwidth]{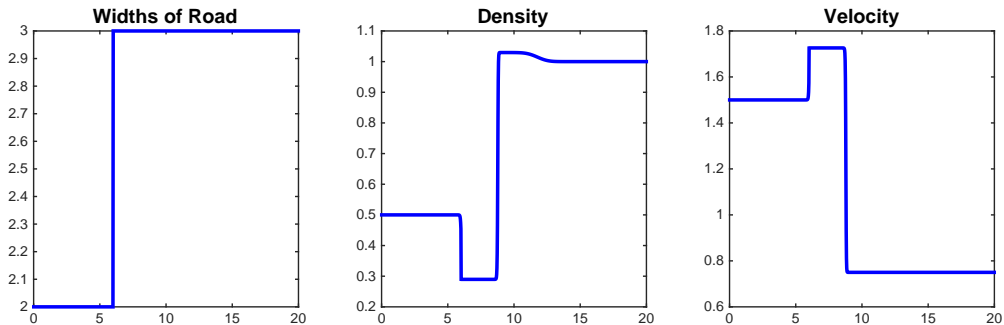}
\end{minipage}
\caption*{Fig. 5.3. $S_0\oplus S\oplus J$.}
\end{figure}

\noindent%
{\bf Example 4.} (see Fig. 5.4)
The initial data is given as

\begin{equation}\label{5.4-1}
U(x,0)=\left\{\begin{array}{ll}
U_-=(3.0,0.75,  2.0)\in D_2,\quad 0<x<6,\\[6pt]
U_+=(8.0, 0.5, 3.5)\in D_2, \quad  6<x<20,
\end{array}\right.\qquad \gamma=3.25.
\end{equation}

The numerical result is shown at time $t=1.0s$ in Fig. 5.4. The solution begins with a stationary wave $S_0$, followed by a rarefaction wave $R$, and then a contact discontinuity $J$. It is the same as in the case 1.2.

\begin{figure}[htbp]
\begin{minipage}[t]{\textwidth}
\centering
\includegraphics[width=\textwidth]{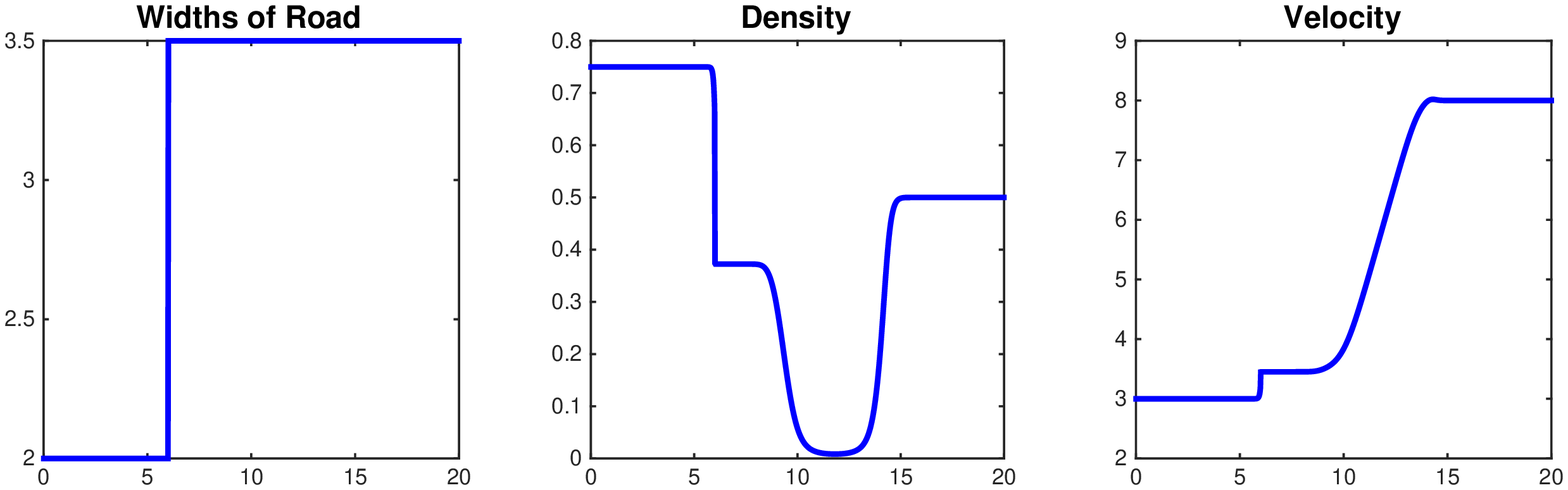}
\end{minipage}
\caption*{Fig. 5.4. $S_0\oplus R\oplus J$.}
\end{figure}

\noindent%
{\bf Example 5.} (see Fig. 5.5)
The initial data is given as

\begin{equation}\label{5.5-1}
U(x,0)=\left\{\begin{array}{ll}
U_-=(4.0,1.0,  2.0)\in D_2,\quad 0<x<10,\\[6pt]
U_+=(0.6,3.0,  2.5)\in D_1, \quad  10<x<20,
\end{array}\right.\qquad \gamma=1.5.
\end{equation}

The numerical result is shown at time $t=7.5s$ in Fig. 5.5. The solution begins with a shock wave $S$, followed by a stationary wave $S_0$, and then a contact discontinuity $J$. The result is the same as in the case 1.4.

\begin{figure}[htbp]
\begin{minipage}[t]{\textwidth}
\centering
\includegraphics[width=\textwidth]{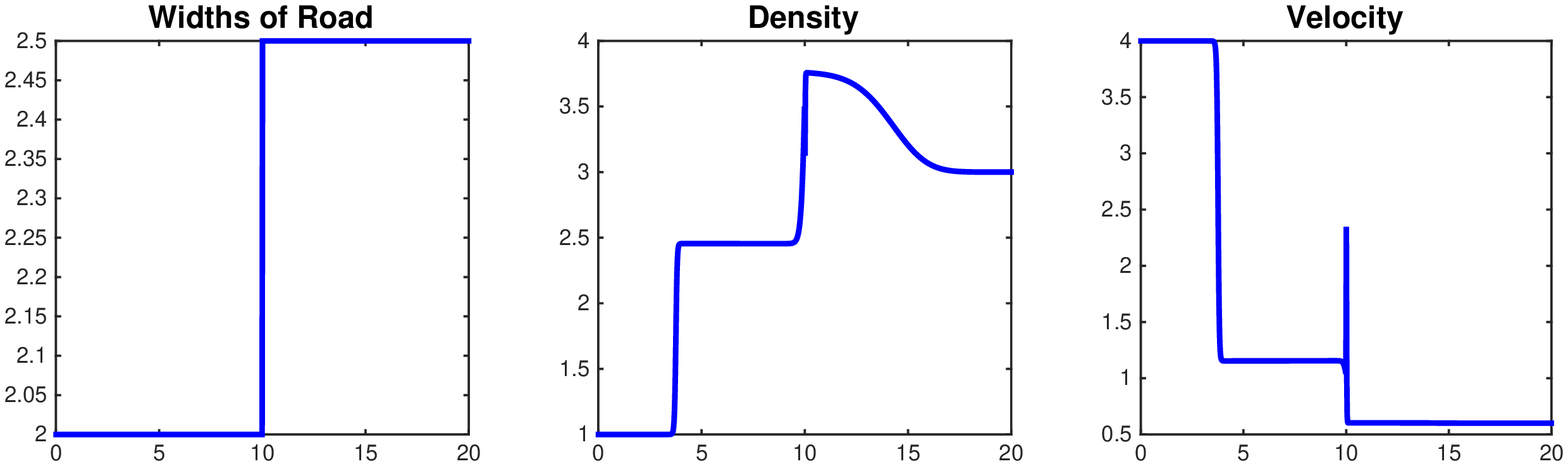}
\end{minipage}
\caption*{Fig. 5.5. $S\oplus S_0\oplus J$.}
\end{figure}

\noindent%
{\bf Example 6.} (see Fig. 5.6)
The initial data is given as

\begin{equation}\label{5.6-1}
U(x,0)=\left\{\begin{array}{ll}
U_-=(1.0,2.0,  2.0)\in D_1,\quad 0<x<10,\\[6pt]
U_+=(1.25,2.0, 2.5)\in D_1, \quad  10<x<20,
\end{array}\right.\qquad \gamma=1.75.
\end{equation}

The numerical result is shown at time $t=1.6s$ in Fig. 5.6. The solution begins with a rarefaction wave $R$, followed by a stationary wave $S_0$, and then a contact discontinuity $J$. The result is the same as in the case 2.4.

\begin{figure}[htbp]
\begin{minipage}[t]{\textwidth}
\centering
\includegraphics[width=\textwidth]{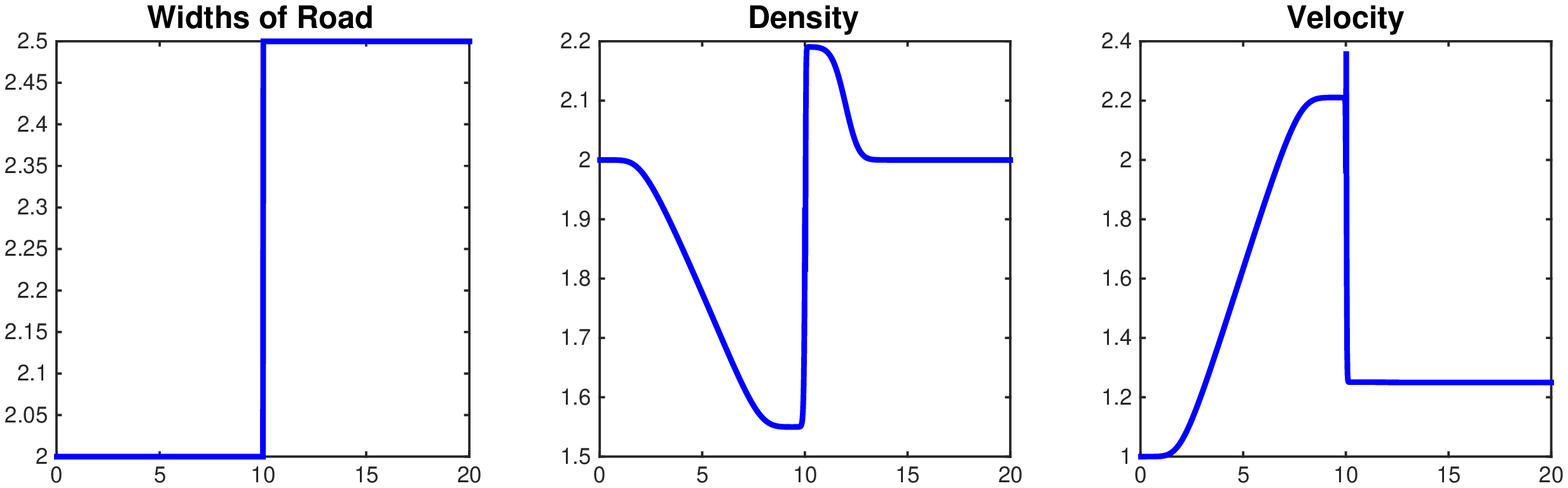}
\end{minipage}
\caption*{Fig. 5.6. $R\oplus S_0\oplus J$.}
\end{figure}

\begin{figure}[htbp]
\begin{minipage}[t]{\textwidth}
\centering
\includegraphics[width=\textwidth]{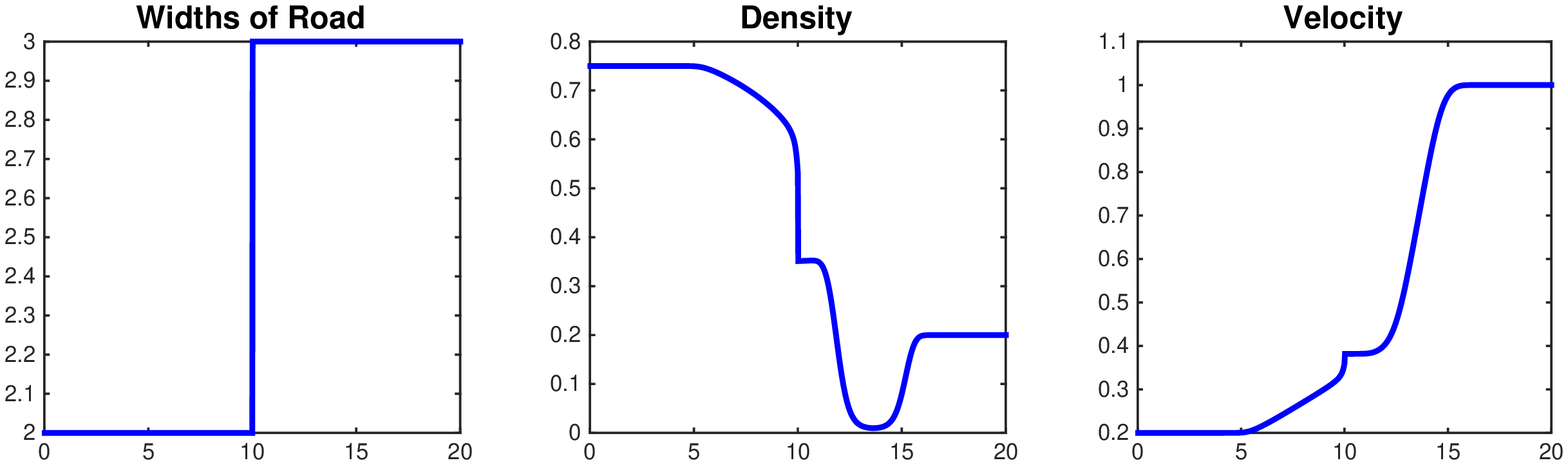}
\end{minipage}
\caption*{Fig. 5.7. $R\oplus S_0\oplus R\oplus J$.}
\end{figure}

\begin{figure}[htbp]
\begin{minipage}[t]{\textwidth}
\centering
\includegraphics[width=\textwidth]{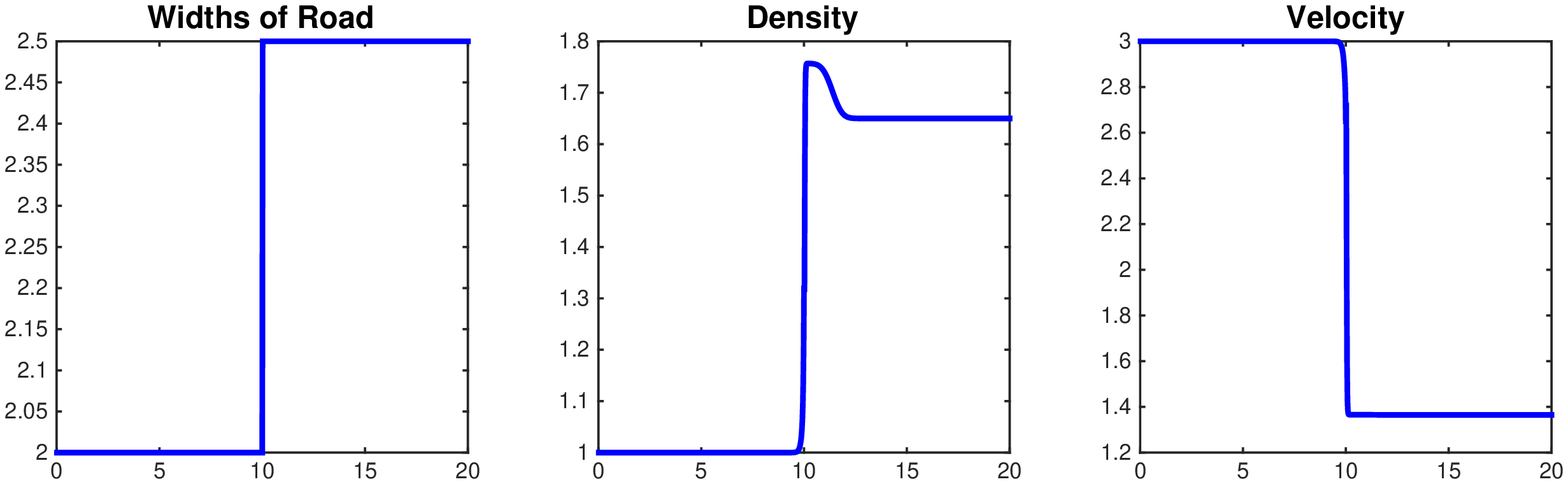}
\end{minipage}
\caption*{Fig. 5.8. $S_0\oplus S \oplus S_0\oplus J$.}
\end{figure}

\noindent%
{\bf Example 7.} (see Fig. 5.7)
The initial data is given as

\begin{equation}\label{5.7-1}
U(x,0)=\left\{\begin{array}{ll}
U_-=(0.2,0.75,  2.0)\in D_1,\quad 0<x<10,\\[6pt]
U_+=(1.0, 0.2, 3.0)\in D_2, \quad  10<x<20,
\end{array}\right.\qquad \gamma=6.0.
\end{equation}

The numerical result is shown at time $t=5.0s$ in Fig. 5.7. The solution begins with a rarefaction wave $R$, which coincide with the stationary wave $S_0$, followed by a rarefaction wave $R$, and then a contact discontinuity $J$. The result is the same as in the case 2.5.

\noindent%
{\bf Example 8.} (see Fig. 5.8)
The initial data is given as

\begin{equation}\label{5.8-1}
U(x,0)=\left\{\begin{array}{ll}
U_-=(3.0,1.0, 2.0)\in D_2,\quad 0<x<10,\\[6pt]
U_+=(1.365, 1.65, 2.5)\in D_1, \quad  10<x<20,
\end{array}\right.\qquad \gamma=2.0.
\end{equation}

The numerical result is shown at time $t=1.0s$ in Fig. 5.8.  The solution includes three waves with the same zero speed, followed by contact discontinuity $J$. Compare this case with example 2, we can see clearly that the stationary wave of Riemann problem is the limit of the standing wave in example 2.

\begin{figure}[htbp]
\begin{minipage}[t]{\textwidth}
\centering
\includegraphics[width=\textwidth]{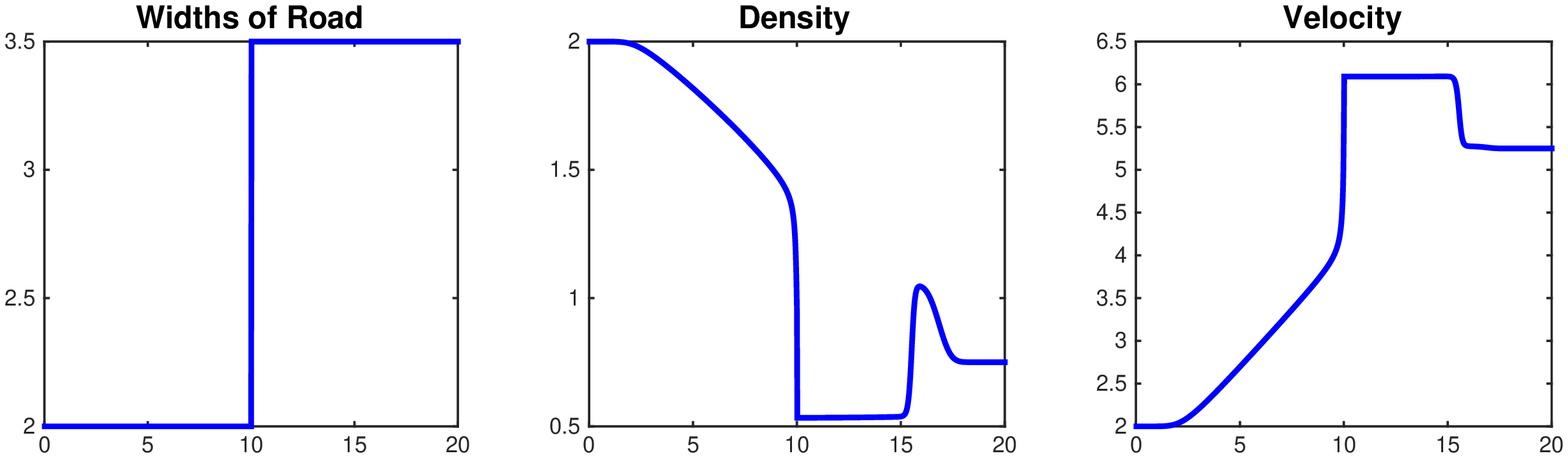}
\end{minipage}
\caption*{Fig. 5.9. $R\oplus S_0\oplus S\oplus J$.}
\end{figure}

\noindent%
{\bf Example 9.} (see Fig. 5.9)
The initial data is given as

\begin{equation}\label{5.9-1}
U(x,0)=\left\{\begin{array}{ll}
U_-=(2.0, 2.0, 2.0)\in D_1,\quad 0<x<10,\\[6pt]
U_+=(5.25, 0.75, 3.5)\in D_2, \quad  10<x<20,
\end{array}\right.\qquad \gamma=2.0.
\end{equation}

The numerical result is shown at time $t=1.25s$ in Fig. 5.9. The solution begins with a rarefaction wave $R$, which is coincide with the stationary wave $S_0$, followed by a shock wave $S$,  and then contact discontinuity $J$. The result is the same as in the case 2.2.

In summary, we establish a model which describes the traffic flow on a road with variable widths. The new model is based on the Aw-Rascle model where the pressure is regarded as the velocity offset. Then we construct the Riemann solutions. To ensure the uniqueness of the steady wave, we impose the global entropy condition on the road widths $a(x)$. The uniqueness and nonuniqueness of the Riemann solution is also discussed. Finally we give some numerical results which are corresponding to the analytical solutions.

\

\appendix

\section{The relative position of $\widetilde{U}_-$ and $\widetilde{U}_{-*}$}

\begin{lem}\label{A1}

Assume that $U_-\in D_2$, $U_{-*}\in S_0(U,U_-)$. From lemma 3.5, there exist $\widetilde{U}_-\in S(U,U_-)$ and $\widetilde{U}_{-*}\in S(U,U_{-*})$, at which both the shock have zero speed. Denote the curves $\widetilde{\Gamma}_-: \tilde{u}_-=c_1\gamma \tilde{\rho}_-^{\gamma}$ and $\widetilde{\Gamma}_{-*}: \tilde{u}_{-*}=c_2\gamma \tilde{\rho}_{-*}^{\gamma}$.
We conclude that $c_1>c_2$, i.e., $\widetilde{U}_-$ is "closer" to $\Gamma_+$ than $\widetilde{U}_{-*}$.

\begin{figure}[htbp]
\begin{minipage}[t]{0.45\textwidth}
\centering
\includegraphics[width=\textwidth]{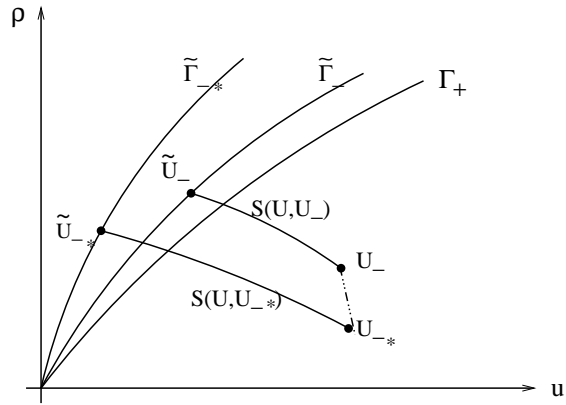}
\end{minipage}
\caption*{Fig.A. The relative position of $\widetilde{U}_-$ and $\widetilde{U}_{-*}$.}
\end{figure}
\end{lem}

 \begin{proof}
From the assumption, we have
\begin{equation*}
\left\{\begin{array}{ll}
\displaystyle \rho_-u_-=\tilde{\rho}_-\tilde{u}_-,\\
u_-+\rho_-^{\gamma}=\tilde{u}_-+\tilde{\rho}_-^{\gamma}.
\end{array}\right.
\end{equation*}
Substituting $\displaystyle \tilde{u}_-=\frac{\rho_-u_-}{\tilde{\rho}_-}$ into the second equation, we have
$$
\tilde{\rho}_-^{\gamma+1}-(u_-+\rho_-^{\gamma})\tilde{\rho}_-+\rho_-u_-=0,
$$
If we denote $u_-=c\gamma\rho_-^{\gamma}~(c>1)$ and $\displaystyle x=\frac{\tilde{\rho}_-}{\rho_-}$, it follows
$$
f(x):=x^{\gamma+1}-(c\gamma+1)x+c\gamma=0,
$$
$f(x)$ admits two solutions $1$ and $x_1>1$. So we verify that $\tilde{\rho}_-=x_1\rho_-$. Moreover, we have $\displaystyle \tilde{u}_-=\frac{c}{x_1^{\gamma+1}} \gamma\tilde{\rho}_-^{\gamma}=c_1\gamma\tilde{\rho}_-^{\gamma}$. It leads to $\displaystyle c=c_1x_1^{\gamma+1}$. \\
Now, to prove lemma A, it is sufficient to prove that $c_1$ changes in the opposite direction with $c$. We substitute $\displaystyle c=c_1x_1^{\gamma+1}$ into $f(x)$, and get
$$
\displaystyle c^{\frac{\gamma}{\gamma+1}}-(c\gamma+1)c_1^{\frac{\gamma}{\gamma+1}}+\gamma c^{\frac{\gamma}{\gamma+1}} c_1=0.
$$
Differential the above equation, we have
$$
\displaystyle \left(\frac{\gamma}{\gamma+1} c^{-\frac1{\gamma+1}}(1+\gamma c_1)-\gamma c_1^{\frac{\gamma}{\gamma+1}}\right){\rm d}c=\left(\frac{\gamma}{\gamma+1}(c\gamma+1)c_1^{-\frac1{\gamma+1}}-\gamma c^{\frac{\gamma}{\gamma+1}}\right){\rm d}c_1.
$$
A direct calculation leads to $\left(\frac{\gamma}{\gamma+1} c^{-\frac1{\gamma+1}}(1+\gamma c_1)-\gamma c_1^{\frac{\gamma}{\gamma+1}}\right)<0$ and $\left(\frac{\gamma}{\gamma+1}(c\gamma+1)c_1^{-\frac1{\gamma+1}}-\gamma c^{\frac{\gamma}{\gamma+1}}\right)>0$. It follows ${\rm d}c_1 \cdot {\rm d}c<0$. Since $\displaystyle \frac{u_-}{\gamma\rho_-^{\gamma}}<\frac{u_{-*}}{\gamma \rho_{-*}^{\gamma}}$ from \eqref{2.24}, we have $\displaystyle \frac{\tilde{u}_-}{\gamma \tilde{\rho}_-^{\gamma}}>\frac{\tilde{u}_{-*}}{\gamma \tilde{\rho}_{-*}^{\gamma}}$, i.e., $c_1>c_2$.

\end{proof}

  \section{The monotonic property of wave curve}

\begin{lem}\label{B1}
If $U_-\in D_2$, $U_{-*}$ and $U_-^{*}$ are denoted in case 1. Assume that $U_0\in S_0(U,U_-)$ with the cross-section area shifting from $a_0$ to $a$, $U_1\in S(U,U_0)$, $U_2\in S_0(U,U_1)$ with the cross-section shifting from $a$ to $a_1$. Such states $U_2$ forms a curve $L: ~U_2(a):~~a\in [a_0,a_1]$. We conclude that $L$ is a decreasing curve. Moreover, if $a_1$ is close to $a_0$, we have\\
1). On the expand road with $a_1>a_0$, we have $\tilde{u}_-^{*}<\tilde{u}_{-*},  \tilde{\rho}_-^{*}>\tilde{\rho}_{-*}$. See Fig.A (left).\\
2). On the contract road with $a_1<a_0$, we have $\tilde{u}_-^{*}>\tilde{u}_{-*},  \tilde{\rho}_-^{*}<\tilde{\rho}_{-*}$. See Fig.A (right).

\begin{figure}[htbp]
\subfigure{
\begin{minipage}[t]{0.4\textwidth}
\centering
\includegraphics[width=\textwidth]{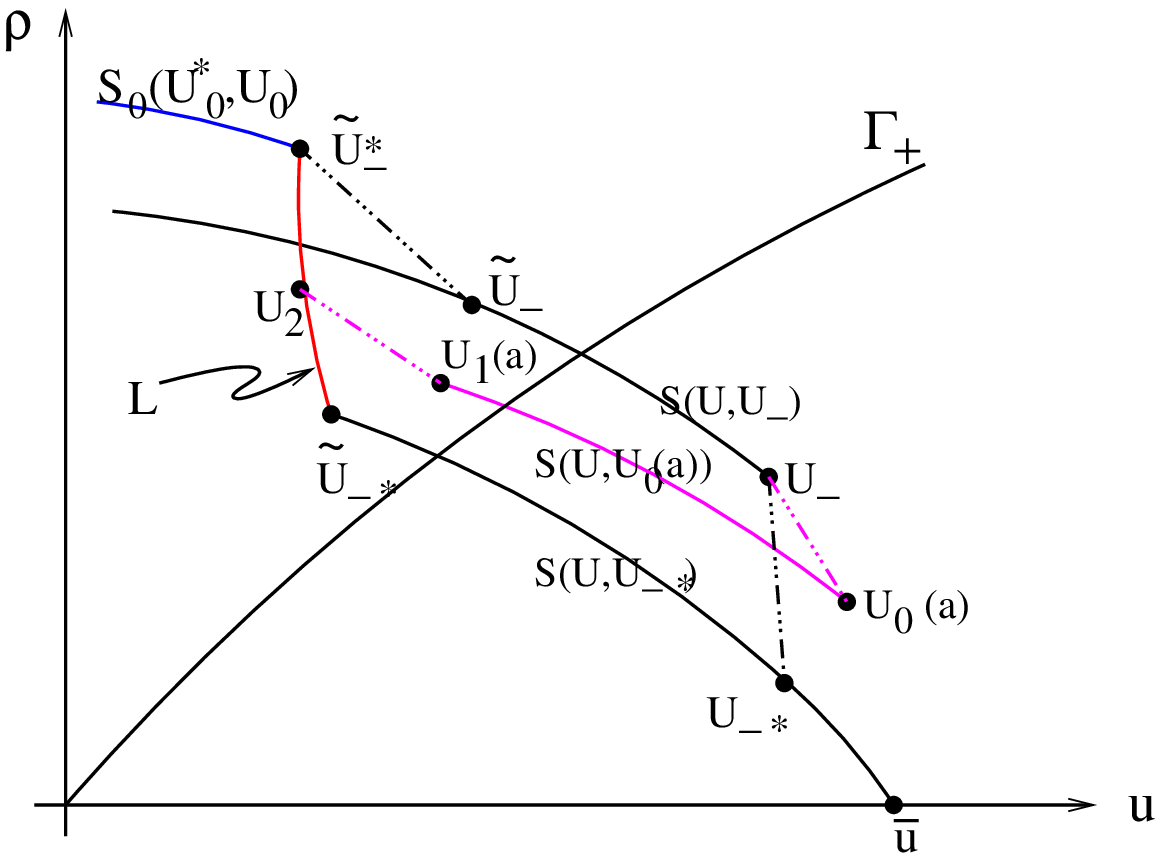}
\end{minipage}
}
\subfigure{
\begin{minipage}[t]{0.4\textwidth}
\centering
\includegraphics[width=\textwidth]{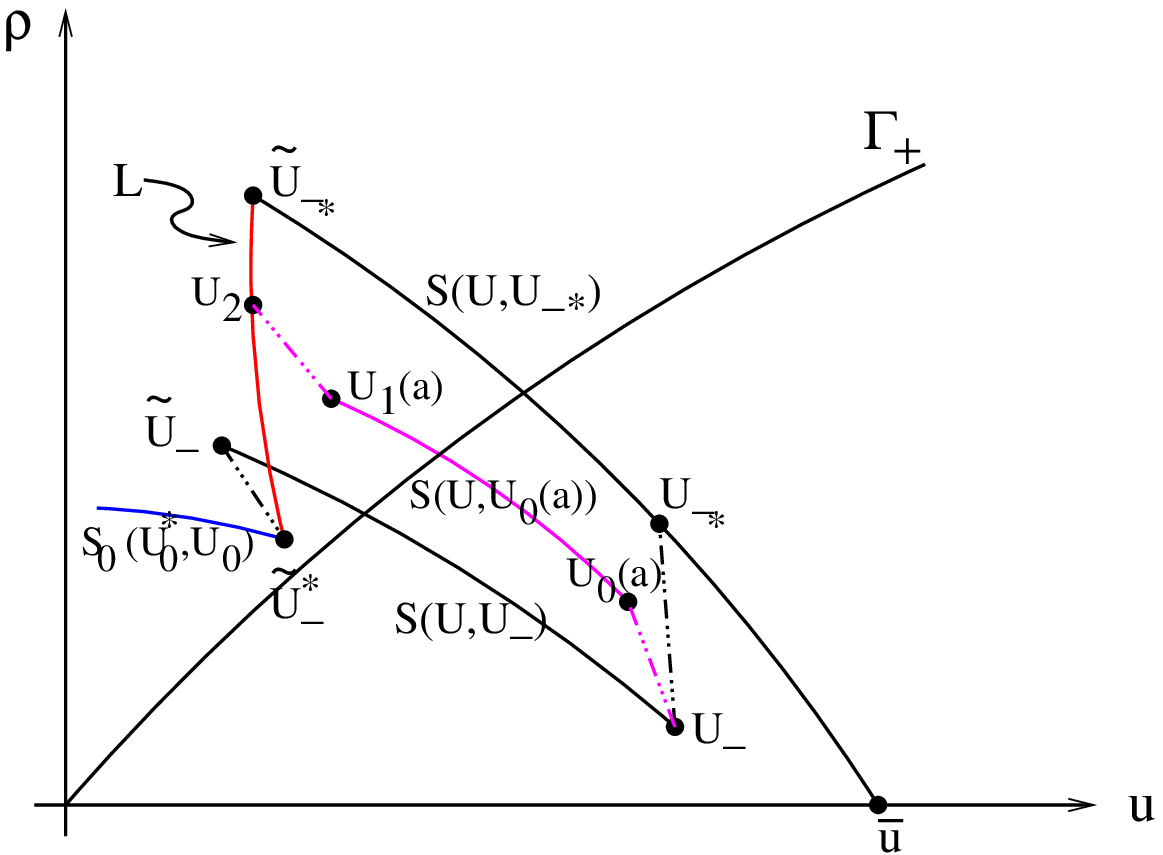}
\end{minipage}
}
\caption*{Fig. B. The curve $L$ with $a_0<a_1$ (left) and  $a_0>a_1$ (right).}
\end{figure}
\end{lem}

 \begin{proof}
From the assumption, we have $U_2(a_0)=\widetilde{U}_{-*}$, $U_2(a_1)=\widetilde{U}_-^{*}$.  Furthermore, we have
\begin{equation}\label{5.4}
\left\{\begin{array}{ll}
S_0(U_0,U_-):
\left\{\begin{array}{ll}
\displaystyle a_0\rho_-u_-=a\rho u,\\[8pt]
\displaystyle u_-^{\frac{\gamma}{\gamma+1}}(\rho_-^{\gamma}+\frac{\gamma}{2\gamma+1})=u_0^{\frac{\gamma}{\gamma+1}}(\rho_0^{\gamma}+\frac{\gamma}{2\gamma+1}),\\
\end{array}\right. \\[8pt]
S(U_1,U_0):
\left\{\begin{array}{ll}
\displaystyle \rho_0u_0=\rho_1u_1,\\[8pt]
\displaystyle u_0+\rho_0^{\gamma}=u_1+\rho_1^{\gamma},
\end{array}\right.\\
S_0(U_2,U_1):
\left\{\begin{array}{ll}
\displaystyle a\rho_1u_1=a_1\rho_2 u_2,\\[8pt]
\displaystyle u_1^{\frac{\gamma}{\gamma+1}}(\rho_1^{\gamma}+\frac{\gamma}{2\gamma+1})=u_2^{\frac{\gamma}{\gamma+1}}(\rho_2^{\gamma}+\frac{\gamma}{2\gamma+1}).
\end{array}\right.
\end{array}\right.
\end{equation}
From \eqref{5.4}, we have $\displaystyle \rho_2u_2=\frac{a_0\rho_-u_-}{a_1}$. So it is enough to judge the sign of $\frac{{\rm d}u_2}{{\rm d}a}$.  Differential the expressions of $S_0(U_0,U_-)$, we have
\begin{equation}\label{5.5}
\left\{\begin{array}{ll}
\displaystyle \frac{{\rm d}\rho_0}{{\rm d} a}=-\frac{\rho_0(u_0+\rho_0^{\gamma})}{a_0(u_0-\gamma\rho_0^{\gamma})},\\[12pt]
\displaystyle \frac{{\rm d}u_0}{{\rm d} a}=(\gamma+1)\frac{u_0\rho_0^{\gamma}}{a_0(u_0-\gamma\rho_0^{\gamma})}.\\
\end{array}\right.
\end{equation}
Differential the expressions of $S(U_1,U_0)$ and using \eqref{5.5}, we have
\begin{equation}\label{5.6}
\left\{\begin{array}{ll}
\displaystyle \frac{{\rm d}\rho_1}{{\rm d} a}=-\frac{\rho_0u_0+\rho_1\rho_0^{\gamma}}{a_0(u_1-\gamma\rho_1^{\gamma})},\\[12pt]
\displaystyle \frac{{\rm d}u_1}{{\rm d} a}=\frac{u_1\rho_0^{\gamma}+\gamma\rho_1^{\gamma-1}\rho_0u_0}{a_0(u_1-\gamma\rho_1^{\gamma})}.\\
\end{array}\right.
\end{equation}
Finally, from the expression of \eqref{5.6}, we have
\begin{equation}\label{5.7}
\displaystyle \frac{{\rm d}u_2}{{\rm d} a}=\frac{u_1}{a_0a_1}\frac{(\gamma+1)(a_0-a)\rho_1\rho_2^{\gamma-1}+a_1(\rho_1^{\gamma}-\rho_0^{\gamma})\left(\frac{u_2}{u_1}\right)^{\frac1{\gamma+1}}}{\gamma\rho_2^{\gamma}-u_2}.
\end{equation}
When $a$ changes monotonically, we have $\displaystyle \frac{{\rm d}u_2}{{\rm d}a}<0$ as $a_0>a_1$. When $a_0<a_1$, $\displaystyle \lim_{a\to a_0}{\frac{{\rm d}u_2}{{\rm d}a}}>0$, $\displaystyle \frac{{\rm d}u_2}{{\rm d}a}$ keeps its sign as $a$ is close to $a_0$. Thus we prove lemma B.

 \end{proof}

\bibliographystyle{amsplain}

 \end{document}